\documentclass[11pt]{article}
\usepackage{amssymb,amsmath,amsthm}

\newtheorem{theorem}{Theorem}[section]
\newtheorem{lemma}[theorem]{Lemma}
\newtheorem{proposition}[theorem]{Proposition}
\newtheorem{corollary}[theorem]{Corollary}
\newtheorem{question}[theorem]{Question}

\theoremstyle{definition}
\newtheorem{definition}[theorem]{Definition}
\newtheorem{example}[theorem]{Example}

\theoremstyle{remark}

\numberwithin{equation}{section}

\newcommand{\quotgp}[2]{\mathchoice
   {\raise.5ex\hbox{$(#1)$}\hbox{\Large$/$}\lower.25ex\hbox{$(#2)$}}
   {\raise.25ex\hbox{\small$(#1)$}/\lower.25ex\hbox{\small$(#2)$}}
  {\raise.15ex\hbox{\small$\scriptstyle(#1)$}/
  \lower.15ex\hbox{\small$\scriptstyle(#2)$}}
{\raise.1ex\hbox{\small$\scriptscriptstyle(#1)$}/ 
\lower.1ex\hbox{\small$\scriptscriptstyle(#2)$}}
  }

\newcommand{\RR}{{\mathbb R}}			
\newcommand{\ZZ}{{\mathbb Z}}			
\newcommand{\NN}{{\mathbb N}}			
\newcommand{\QQ}{{\mathbb Q}}			
\newcommand{\TT}{{\mathbb T}}
\newcommand{\G}{\Gamma}
\newcommand{\g}{ {\mathfrak g} }

\newcommand{\mcal}{\mathcal}			

\newcommand{\mB}{{\mcal B}}

\newcommand{\mA}{{\mcal A}}

\newcommand{\sm}{\setminus}			

\newcommand{\e}{\varepsilon}

\renewcommand{\and}{\text{ and }}

\newcommand{\nrml}{\triangleleft}

\newcommand{\EQ}{\begin{eqnarray*}}
\newcommand{\EQE}{\end{eqnarray*}}

\newcommand{\ENUM}{\begin{enumerate}}
\newcommand{\ENUME}{\end{enumerate}}

\title{Equidistribution of singular measures on nilmanifolds and skew products}
\author{Fabrizio Polo}

\begin{document}

\maketitle

\begin{abstract}
We prove that for a minimal rotation $T$ on a 2-step nilmanifold and any measure $\mu$, the push-forward $T^n_\star \mu$ of $\mu$ under $T^n$ tends toward Haar measure if and only if $\mu$ projects to Haar measure on the maximal torus factor.  For an arbitrary nilmanifold we get the same result along a sequence of uniform density $1.$  These results strengthen Parry's result \cite{parry} that such systems are uniquely ergodic. Extending the work of Furstenberg \cite{furstenberg}, we get the same result for a large class of iterated skew products. Additionally we prove a multiplicative ergodic theorem for functions taking values in the upper unipotent group.  Finally, we characterize limits of $T^n_\star \mu$ for some skew product transformations with expansive fibers.  All results are presented in terms of twisting and weak twisting, properties which strengthen unique ergodicity in a way analogous to how mixing and weak mixing strengthen ergodicity for measure preserving systems.
\end{abstract}


\section{Introduction}

By a {\it topological dynamical system} we shall mean a compact metric space $X$ equipped with a homeomorphism $T:X \to X$ (all systems will be assumed invertible.)  We will denote such a system by a pair $(X,T).$  Given two systems, $(X,T)$ and $(Y,S)$, a {\it factor map} $\Phi:(X,T) \to (Y,S)$ is a surjective continuous map $\Phi:X \to Y$ such that $\Phi(T(x)) = S( \Phi(y)).$  Consider the following easy rephrasing of well known results from ergodic theory.  (For the basic definitions of ergodic theory, see \cite{walters}.)
\begin{proposition} \label{mixProp}
 Let $(X, T)$ be a topological dynamical system and let $m$ be an invariant probability measure on $X.$
\begin{enumerate}
\item $m$ is ergodic if and only if for all $\mu,$ absolutely continuous with respect to $m,$ and for all $f \in C(X)$
\[ \lim_{N \to \infty} \frac{1}{N} \sum_{n =0}^{N-1} \int_X f \circ T^n d \mu - \int_X f dm = 0. \]
\item $m$ is weakly mixing if and only if for all $\mu,$ absolutely continuous with respect to $m,$ and for all $f \in C(X)$
\[ \lim_{N \to \infty} \frac{1}{N} \sum_{n =0}^{N-1} \left| \int_X f \circ T^n d \mu - \int_X f dm \right| = 0. \]
\item $m$ is mixing if and only if for all $\mu,$ absolutely continuous with respect to $m,$ and for all $f \in C(X)$
\[ \lim_{n \to \infty} \int_X f \circ T^n d \mu - \int_X f dm = 0. \]
\end{enumerate}
\end{proposition}
\begin{proof}
Let $d \mu = \psi dm$ where $\psi \in L^1(m)$ and let $\psi_M = \min( \psi, M).$  Then $f,\psi_M \in L^2(m),$ so in statements 1,2, and 3, the Hilbertian definitions of ergodicity, weak mixing, and mixing dictate that the appropriate limits hold when
\[ \int f \circ T^n d \mu \text{ \ is replaced by } \int f \circ T^n \psi_M dm = \left< f \circ T^n, \psi_m \right>. \]
Now we use the fact that $f$ is bounded to pass to a limit in $M$ and derive the same results for $\int f \circ T^n d \mu.$

For the converses, we observe $L^1(m)$ contains $L^2(m)$ and $C(X)$ is dense in $L^2(m).$  The Hilbertian definitions follow immediately from the limits above.
\end{proof}
The goal of this paper is to study similar averages and limits where $\mu$ has been replaced by some probability measure that is singular with respect to $m.$  More specifically, if $(X,T)$ is a topological system with unique invariant measure $m,$ we study those $\mu$ for which we could expect the limits in Proposition \ref{mixProp} to hold.

In Proposition \ref{mixProp}, using test functions $f$ from $C(X)$ seems unnatural since the topology of $X$ is irrelevant to the usual definitions of ergodicity, weak mixing, and mixing.  One usually makes similar statements in $L^2(m)$ where the integral $\int f d\mu$ against an absolutely continuous measure $d \mu = \varphi dm$ is replaced by an inner product: $\int f d \mu = \int f \varphi dm = \left< f, \bar \varphi \right>.$  However, when one wishes to study singular $\mu,$ there is no obvious analogue of $\varphi$, so $L^2(m)$ is insufficient.

One can rephrase Proposition \ref{mixProp} in terms the space $P(X)$ of Borel probability measures on $X.$  Considered as a subspace of $C(X)^\star$, $P(X)$ may be equipped with the weak$^\star$ topology.  In particular Proposition \ref{mixProp} (3) can be rewritten: $m$ is mixing if and only if for all $\mu$ absolutely continuous with respect to $m, \lim_{n \to \infty} T^n_\star \mu = m,$ where $T^n_\star \mu$ is the push-forward measure defined by $\int f dT^n_\star \mu = \int f \circ T^n d \mu.$

Before stating our results regarding singular $\mu,$ we recall some preliminaries.  A {\it nilmanifold } is a space $X$ of the form $G / \G$ where $G$ is a connected nilpotent Lie group and $\Gamma$ is a cocompact lattice.  
That is, $\G$ is a discrete subgroup of $G$ such that the quotient space $G/\G$ is compact.  $G$ acts on $X$ by left multiplication.
Such groups $G$ admit a bi-invariant {\it Haar measure } $m.$  By identifying $X$ with a fundamental domain for the left action of $G$ on itself, we can equip $X$ with a finite measure.  Since $m$ is unique up to scaling, we may assume the measure on $X$ is a probability measure.  For simplicity we write $m$ for the measure $G$ on $X$ and call both Haar measures.  The left action of $G$ on $X$ preserves $m.$  We say a sequence of measures $\mu_n$ on $X$ {\it equidistributes } if for all $f \in C(X),$
\[ \lim_{n \to \infty} \int_X f \mu_n = \int f d m. \]

Write $G_0 = G$ and $G_{n+1} = [G, G_n].$  Then $G_{n+1} \nrml G_n,$ and $G_n / G_{n+1}$ is abelian and connected.  The sequence of subgroups $G_n$ is called the {\it lower central series} of $G.$  Since $G$ is nilpotent there exists some $d$ such that $G_d = \{ 1 \}.$  The least such $d$ the called the {\it degree of nilpotency } of $G$ (and $X$.)  The lower central series of $G$ gives us a sequence of quotients $X_n := G / G_n \G$ of $X.$  The fiber above each point in the factor map $X_{n+1} \to X_n$ is homeomorphic to
\[ G_n / G_{n+1} \G \cong (G_n / G_{n+1} ) / ( (\G \cap G_n ) G_{n+1} / G_{n+1} ), \]
which is a compact quotient of $G_n / G_{n+1}$ and hence a torus.  In fact $X_{n+1}$ is a torus bundle over $X_n.$  Therefore $X$ is derived from the one point space $X_d$ by repeatedly taking circle bundles.
We call $X_1$ the {\it maximal torus factor} of $X.$

Let $J$ be a set of integers.  When the following limits exist
\[ d(J) := \lim_{N \to \infty} \frac{\# ( J \cap [1,N]) }{ N } \text{ \ and \ }  d^\star (J) := \lim_{N-M \to \infty} \frac{\# ( J \cap [M,N]) }{ N-M+1 } \]
we refer to their values as the {\it density} and {\it uniform density} of $J.$  If $d^\star (J)$ is well defined then so is $d(J)$ and the two coincide.

\begin{theorem} \label{nilThm}
Let $X = G / \G$ be a nilmanifold.  Fix $u \in G$ and let $T(g\G) = ug \G.$  If the system $(X, T)$ is transitive (i.e. has a dense orbit) then for any probability measure $\mu$ on $X$ that projects to Haar measure on the maximal torus factor, there exists a subset $J \subset \ZZ$ of uniform density zero such that for any $f \in C(X)$
\[ \lim_{n \to \infty, \ n \notin J } \int_X f \circ T^n d\mu = \int_X f dm. \]
In other words, the sequence $\{ T^n_\star \mu : n \notin J \}$ equidistributes.

Furthermore, if one fixes $f$ and lets $\mu$ range then the following limit converges uniformly in $\mu.$
\[ \lim_{N-M \to \infty} \frac{1}{N-M} \sum_{i=M}^{N-1} \left| \int f \circ T^i d \mu - \int f dm \right| = 0. \]
\end{theorem}

When we say $\mu$ projects to Haar measure on the maximal torus factor we mean that $\pi_\star \mu$ is Haar measure on $X_1,$ where $\pi: X \to X_1$ is the obvious factor map.

\begin{theorem} \label{nil2Thm}
Let $(X, T)$ and $\mu$ be as in Theorem \ref{nilThm}.  If we assume that $G$ is a 2-step nilpotent group, then the
limit holds with no exceptional set $J$:
\[ \lim_{n \to \infty} \int_X f \circ T^n d\mu = \int_X f dm. \]
In other words, $T^n_\star \mu$ equidistributes.
\end{theorem}

As will be discussed in the next section, the conclusions of Theorems \ref{nilThm} and \ref{nil2Thm} imply unique ergodicity.  So, these theorems strengthen the result of Parry \cite{parry} which asserts that such systems are uniquely ergodic.

While nilmanifolds (spaces $G/\G$ as in Theorem \ref{nilThm}) are not usually Tori, their topologies are locally similar because, as mentioned above, nilmanifolds can be constructed by repeatedly taking circle bundles.  Combining techniques used in the proof of Theorem \ref{nilThm} with a multiplicative ergodic theorem (Theorem \ref{unipotentcocycleThm}) for unipotent valued cocycles, we get

\begin{theorem} \label{skewThm}
Let $X=\TT^d$ and define $T:X \to X$ by 
\[ T(x_1, \dots, x_d) = (x_1+ \alpha, x_2+f_1(x_1), \dots x_d + f_{d-1}(x_1, \dots, x_{d-1})), \]
where $\alpha$ is irrational and each $f_k(x_1, \dots, x_{k-1}, \cdot) :\TT \to \TT$ is Lipschitz and homotopically non-trivial.
For any probability measure $\mu$ on $X$ that projects to Lebesgue measure on the first coordinate there exists a subset $J \subset \ZZ$ of uniform density zero such that for any $f \in C(X)$
\[ \lim_{n \to \infty, \ n \notin J } \int_X f \circ T^n d\mu = \int_X f dm. \]
In other words, the sequence $\{ T^n_\star \mu : n \notin J \}$ equidistributes.

Furthermore, if one fixes $f$ and lets $\mu$ range then the following limit converges uniformly in $\mu.$
\[ \lim_{N-M \to \infty} \frac{1}{N-M} \sum_{i=M}^{N-1} \left| \int f \circ T^i d \mu - \int f dm \right| = 0. \]
\end{theorem}

\begin{theorem} \label{skew2Thm}
Let $(X, T)$ be as in Theorem \ref{skewThm} with $d=2$ (i.e. $T(x,y) = (x+\alpha, y+f(x))$.)  Then
the limit holds with no exceptional set:
\[ \lim_{n \to \infty} \int_X f \circ T^n d\mu = \int_X f dm. \]
\end{theorem}
Furstenberg proves in \cite{furstenberg} (Theorem 2.1) that such systems are uniquely ergodic.  Our theorem is a direct extension of his.

The multiplicative ergodic theorem alluded to above is one of the most significant results in this paper, so we include it here.  Let $U$ be the group of upper triangular $d \times d$ matrices with entries in $\RR.$  Then $U$ is admits a one parameter family $\theta_t$ of dilations given by
\[ \theta_t \left( \begin{array}{ccccc}
1 & u_{1,2} & u_{1,3} & \cdots & u_{1,d} \\
0 & 1 & u_{2,3} & \ & u_{2,d} \\
0 & 0 & 1 & \ & \vdots \\
\vdots & \ & \ & \ddots  & \ \\
0 & 0 & 0 & \ & u_{d-1,d} \\
0 & 0 & 0 & \cdots & 1
\end{array} \right)
=
\left( \begin{array}{ccccc}
1 & t u_{1,2} & t^2 u_{1,3} & \cdots & t^{d-1} u_{1,d} \\
0 & 1 & t u_{2,3} & \ & t^{d-2} u_{2,d} \\
0 & 0 & 1 & \ & \vdots \\
\vdots & \ & \ & \ddots & \ \\
0 & 0 & 0 & \ & t u_{d-1,d} \\
0 & 0 & 0 & \cdots & 1
\end{array} \right)
\] 
More formally $(\theta_t(u))_{i,j} = t^{j-i} u_{i,j}$ for $j \geq i.$  It is not hard to check that each $\theta_t$ is an automorphism of $U.$  In fact $\theta: t \mapsto \theta_t$ is a homomorphism from the semigroup $((0,\infty), \times)$ into
the automorphism group of $U.$  If we equip the latter group with the topology of uniform convergence on compact sets,  then $\theta$ is continuous.

\begin{theorem} \label{unipotentcocycleThm}
Let $(X, \mB, m, T)$ be a probability measure preserving system and suppose $f:X \to U$ is bounded and measurable.  Then there exists $f^*: X \to U$ (also bounded, measurable) such that for almost every $x \in X$
\[ \lim_{n \to \infty} \theta_{1/n} ( f ( T^{n-1} x) \cdots f(T^2 x) f (Tx) f(x) ) = f^*(x).\]
Furthermore, $f^* \circ T = f^*$ almost everywhere.

If $m$ is ergodic then, almost everywhere, 
\[ f^*_{i,j} = \lambda \prod_{k=i}^{j-1} \, \int_X f_{k,k+1} dm, \]
where $\lambda = \lambda(j-i)$ is some positive constant depending only on $j-i$.
\end{theorem}

Just as Birkhoff's point-wise theorem becomes stronger in the uniquely ergodic case, so does Theorem \ref{unipotentcocycleThm}.

\begin{theorem} \label{ueunipotentcocycleThm}
Suppose $(X,T)$ is uniquely ergodic and $f:X \to U$ is continuous.  Then
\[ \lim_{n \to \infty} \theta_{1/n} ( f ( T^{n-1} x) \cdots f(T^2 x) f (Tx) f(x) ) \]
converges uniformly to the constant given in Theorem \ref{unipotentcocycleThm}.
\end{theorem}

Finally we prove some partial results on the behavior of measures under pushforward in systems like $(x+\alpha, 
2y+f(x))$ which are neither isometric extensions, nor iterated isometric extensions over their maximal equicontinuous factor, and hence are of a fundamentally different character from the systems discussed above.  These theorems prove the existence of some interesting $T_\star$-invariant probability measures on $P(X),$ which, as we will see in the next section, are pertinent to the asymptotic behavior of $T^n_\star \mu.$


\section{Twisting and weak twisting} \label{twisting}

Before proceeding to the proofs, we introduce some new definitions (twisting and weak twisting) which provide a nice abstract perspective on the results.  We wish to build an analogy between the triples ergodicity, weak mixing, strong mixing, and unique ergodicity, weak twisting, twisting.

Suppose $(X,T)$ is a transitive topological system (i.e. has a point with dense orbit.)  We can define $T_\star: P(X) \to P(X)$ by the pushforward:
\[ T_\star \mu (A) := \mu(T^{-1} A), \text{ \ or equivalently \ }
\int f d T_\star \mu := \int f \circ T d \mu. \]
This makes $(P(X),T_\star)$ into a compact topological system.  Portions of this system have been studied by Glasner in \cite{glasner}, who introduced measure theoretic quasi-factors (certain invariant probability measures on $P(X).$  This avenue of research was furthered by Glasner and Weiss in \cite{glasnerweiss}.

An invariant measure $m$ on $X$ appears in this system as a fixed point. We will study the basin of attraction of $m$ (that is, the set of $\mu$ for which $T^n_\star \mu \to m.$   As explained in Proposition \ref{mixProp}, $m$ is mixing if and only if every probability measure $\mu$ absolutely continuous with respect to $m$ attracts to $m$.  So, if $m$ has full support and is mixing then its basin of attraction is dense.
If all of $P(X)$ attracts to $m$ then, in particular $\delta_x$ attracts to $m$ for each $x \in X.$  But this implies that $m$ is itself a point mass $\delta_{x_0}.$  It follows that $T^n(x)$ must tends to $x_0$ for every $x.$  From most perspectives, this is not very interesting.  Indeed, from a measure theoretic perspective, $(X,T)$ is equivalent to the one point system.

What, then, is the largest closed invariant subset $P'$ of $P(X)$ for which it is reasonable to ask if all $\mu$ in $P'$ attract to $m$?  Notice that a convex combination of measures attracting to $m$ also attracts to $m.$  So we may as well consider only convex $P'.$

Let $K = K (X,T)$ be the maximal equicontinuous factor of $(X,T).$  Specifically, $K$ is the maximal ideal space of the algebra
\[ \{ f \in C(X) : \overline{ \{ f \circ T^n : n \in \ZZ \} } \text{ \ is compact } \}. \]
We will also denote the induced transformation on $K$ by $T.$
The above definition is convenient for our purposes (for instance it makes it clear that $K$ is a functor.)  For other definitions and further discussion, see \cite{glasner2}.  One important fact we will use is that a metric may be chosen for $K(X,T)$ such that $T$ acts by isometries.  Transitivity, equicontinuity, and invertibility together imply minimality.  So $K$ is a compact abelian group and $T(x) = \alpha x$ for some fixed $\alpha \in K$ (see for instance \cite{walters}).

Write $m_K$ for normalized Haar measure on $K.$  Since $m$ was assumed to be invariant on $X,$ it projects to an invariant measure on $K.$  Unique ergodicity of $(K,T)$ tells us $m$ must project to $m_K.$  Suppose some $\mu \in P'$ projects to a measure $\nu$ on $K$ different from $m_K.$   It is easy to see $\nu$ does not attract to $m_K$ (later we will put a metric on $P(K)$ with the property that $T_\star$ is an isometry, so the distance from $T^n_\star \nu$ to $m_K$ is independent of $n.$)  Since $\nu$ does not attract to $m_K,$ it follows that $\mu$ does not attract to $m.$  So, for our purposes, we need only consider $P'$ contained in the following set.
\begin{definition}  We write $P_1 = P_1(X,T)$ for the set of all probability measures on $X$ which project to $m_K.$  This is a convex closed nonempty invariant subset of $P(X).$  So, $(P_1,T_\star|_{P_1})$ is a topological dynamical system.
\end{definition}
\begin{definition}
We will call $(X,T)$ {\it twisting} if
\[ \text{ for all \ } \mu \in P_1, \lim_{n \to \infty} T^n_\star(\mu) = m. \]
\end{definition}

Loosely speaking, a system is twisting if every probability measure which conceivably could, equidistributes under repeated application of $T_\star$ (i.e. unless it is prohibited from doing so by the maximal equicontinuous factor.)

In the spirit of treating $(P(X),T_\star)$ as a topological system, it is interesting to study its invariant measures.  We know of one invariant measure: $\delta_m.$
The argument we will use to prove Proposition \ref{twistingProp} part 2 shows that $m$ is weakly mixing if and only if every measure $\mu$ absolutely continuous with respect to $m$ is a typical point for $\delta_m.$  That is, for any $F \in C(P(X))$
\[ \lim_{N \to \infty} \frac{1}{N} \sum_{n=0}^{N-1} F( T^n_\star \mu) = \int F d \delta_m = F(m). \]
So, if $m$ has full support and is weakly mixing, then there is a dense set of measures $\mu \in P(X)$ which are typical points for the invariant measure $\delta_m$ on the system $(P(X), T_\star).$  Is it possible that every $\mu$ is a typical point for $\delta_m?$  Equivalently, is $(P(X),T_\star)$ ever uniquely ergodic?

The answer is obviously no.  Notice that $(X,T)$ is a subsystem of $(P(X),T_\star)$ where the inclusion is given by $\iota(x) = \delta_x.$  Our invariant measure $m$ on $X$ gives us an invariant measure on $P(X)$ in two ways.  Certainly we have $\delta_m.$ But we also have $\iota_\star(m) =  \iota_\star( \int_X \delta_x dm) = \int_X \delta_{\delta_x} dm.$  Unless $X$ is a single point, these are distinct.  In other words, the only way $(P(X), T_\star)$ can be uniquely ergodic is if $(X,T)$ is already the trivial system.

By analogy to the discussion of attracting fixed points, one might wonder what is the largest subsystem $P' \subset P(X)$ for which it is reasonable to ask if every point is typical for $\delta_m?$  Such measures don't necesarilly tend to $m,$ but they do spend most of their time near $m$ in the sense of uniform density (see Corollary  \ref{wktwistingCor}.)  Equivalently, what is the largest subsystem $P' \subset P(X)$ for which it is reasonable to ask if $(P', T_\star)$ is uniquely ergodic?

Suppose some $\mu \in P'$ projects to $\nu \neq m_K.$  Then $T^n_\star \nu$ avoids some weak$^\star$ neighborhood of $m_K.$  If we average along the orbit of $\nu$ as in the usual proof of the existence of invariant measures (that is, we take a weak$^\star$ limit of $N^{-1} \sum_{n=0}^{N-1} \delta_{T^n_\star \nu}$,) we get a measure different from $\delta_{m_K}.$  Indeed, the two measures can not be the same because they have disjoint support.  It follows that any weak$^\star$ limit of
$N^{-1} \sum_{n=0}^{N-1} \delta_{T^n_\star \mu}$ is different from $\delta_m.$  In particular $(P', T_\star)$ is not uniquely ergodic.  As before, we conclude that, at the very least, we must require $P' \subseteq P_1.$

\begin{definition}
We call $(X,T)$ {\it weakly twisting} if $(P_1, T_\star)$ is uniquely ergodic.  This is equivalent to requiring that each $\mu \in P_1$ is a typical point for $\delta_m.$
\end{definition}

The next proposition follows immediately from the definitions given above.
\begin{proposition} \label{twistingdescriptionProp}
\begin{enumerate}
\item Unique ergodicity of $(X,T)$ is equivalent to the existence of a unique fixed point in $P(X)$ (or in $P_1(X).$)
\item Weak twisting is equivalent to the existence of a unique fixed point in $P(P_1(X))$ (or in $P_1(P_1(X)).$)
\item twisting is equivalent to the existence of a unique universal attracting fixed point in $P_1(X).$
\end{enumerate}
\end{proposition}

\begin{proposition} \label{twistingImplicationProp}
Twisting implies weak twisting.  Weak twisting implies unique ergodicity.
\end{proposition}
\begin{proof}
If $T^n_\star \mu$ tends to $m$ then $\mu$ is a typical point for $\delta_m.$  This gives the first implication.
Since invariant measures $m_1, m_2$ on $X$ gives rise to invariant measures $\delta_{m_1}, \delta_{m_2}$ on $P_1(X),$ we see that unique ergodicity of $P_1$ implies unique ergodicity of $X.$  This proves the second implication.
\end{proof}

As we said at the beginning of this section, we wish to build an analogy between the triples ergodicity, weak mixing, mixing, and unique ergodicty, weak twisting, twisting.
The following characterizations should make that analogy clear (compare to Proposition \ref{mixProp}).
\begin{proposition} \label{twistingProp}
 Let $(X, T)$ be a topological dynamical system and let $m \in P(X)$ be an invariant measure.
\begin{enumerate}
\item $(X,T)$ is uniquely ergodic if and only if for all $\mu \in P_1$ and for all $f \in C(X)$
\[ \lim_{N \to \infty} \frac{1}{N} \sum_{n =0}^{N-1} \int_X f dT^n_\star \mu - \int_X f dm = 0. \]
\item $(X,T)$ is weakly twisting if and only if for all $\mu \in P_1$ and for all $f \in C(X)$
\[ \lim_{N \to \infty} \frac{1}{N} \sum_{n =0}^{N-1} \left| \int_X f dT^n_\star \mu - \int_X f dm \right| = 0. \]
\item $(X,T)$ is twisting if and only if for all $\mu \in P_1$ and for all $f \in C(X)$
\[ \lim_{N \to \infty} \int_X f dT^N_\star \mu - \int_X f dm = 0. \text{ \ That is, \ } \lim_{N \to \infty} T^N_\star \mu = m. \]
\end{enumerate}
\end{proposition}
The only difference between Propositions \ref{mixProp} and \ref{twistingProp} is that, in each case, the assumption that $\mu$ is absolutely continuous with respect to $m$ has been replaced by the assumption that $\mu$ lie in $P_1.$  We should expect that usually $P_1$ contains many measures which are not absolutely continuous with respect to $m.$

It would probably be beneficial for the reader to keep in mind the simple motivating example $T(x,y) = (x+\alpha, y+x)$ on $\TT^2.$  Here, $\alpha$ is some irrational number.  Two dimensional Lebesgue measure is the unique invariant measure for this minimal system (see \cite{furstenberg}.)  The class $P_1$ contains many singular measures.  It includes, for instance, one dimensional Lebesgue measure supported on a horizontal line $\{(x,y_0) : x \in \TT \}.$
Using harmonic analysis, it is not difficult to prove that this measure equidistributes under repeated application of $T.$  We leave this proof to the reader, and derive the result, instead, from the more complicated, but significantly more general Theorem \ref{skew2Thm}.

One may object that in Proposition \ref{twistingProp} (1) the assumption that $\mu$ lie in $P_1$ is not necessary.  However, for parts (2) and (3) it is obvious that this assumption is unavoidable (this is the content of the discussion above involving reasonable choices of $P'$.)  In part (1) we assume $\mu \in P_1$ to reinforce the similarity between Proposition \ref{twistingProp} and \ref{mixProp}. 

In the proof of Proposition \ref{twistingProp} and in future propositions we call on the following well known fact.  The proof is easy and is left to the reader.
\begin{lemma} \label{cesaroLem}
A sequence $x_n$ of non-negative real numbers satisfies
\[ \lim_{N \to \infty} \frac{1}{N} \sum_{n=1}^{N} x_n = 0 \text{ \ if and only if \ } \lim_{n \to \infty, n \in J} x_n = 0 \]
for some $J \subset \NN$ of density $1.$  Similarly,
\[ \lim_{N-M \to \infty} \frac{1}{N-M+1} \sum_{n=M}^{N} x_n = 0 \text{ \ if and only if \ } \lim_{n \to \infty, n \in J} x_n = 0 \]
for some $J \subset \ZZ$ of uniform density $1.$
\end{lemma}

\begin{proof}[Proof of Proposition \ref{twistingProp}.]
Assume $m$ is the unique invariant measure on $(X,T).$  It is well known that this is equivalent to assuming that
\[ \lim_{N \to \infty} \frac{1}{N} \sum_{n =0}^{N-1} f(T^n x) - \int_X f dm = 0 \]
for all $x \in X$ and all $f \in C(X).$  Integrating with respect to $\mu$ and applying the dominated convergence theorem yields the convergence of the average in (1).  Conversely, suppose this average converges for all $\mu \in P_1.$  All invariant measures lie in $P_1.$  So, in particular, this average converges if $\mu$ is invariant.  In this case the expression immediately degenerates into $\int f d\mu = \int f dm$ which implies $\mu = m.$

Now we prove (2).  Assume $(X,T)$ is weakly twisting.  Fix $f \in C(X)$ and define $F \in C(P_1)$ by $F(\mu) = | \int f d\mu - \int f dm |.$  Since $\delta_m$ is the unique invariant measure on $(P_1, T_\star)$ we see that for any $\mu \in P_1,$ the limit in (2) is equal to
\[ \lim_{N \to \infty} \frac{1}{N} \sum_{n=0}^{N-1} F(T^n_\star \mu) = \int F d \delta_m
= F(m) = 0, \]

To see the converse, notice that if the limit in (2) holds then for any $f$ and any $\e>0$ then by Lemma \ref{cesaroLem} we can find a sequence of density $1$ along which $| \int f dT^n_\star \mu - \int f dm | < \e.$ Intersecting finitely many such sequences proves that the set of $n \in \NN$ for which $T^n_\star \mu$ lies in a given weak$^\star$ neighborhood of $m$ has density $1.$  Obviously, then $\mu$ is a typical point for the invariant measure $\delta_m$ on $P_1.$

For (3) there is nothing to prove.  This is the definition.
\end{proof}

The numerous characterizations of unique ergodicity give rise to characterizations of weak twisting.
\begin{corollary} \label{wktwistingCor}
With $(X,T)$ and $m$ as in Proposition \ref{twistingProp}, weak twisting is equivalent to the assertion that for each $\mu \in P_1$ there exists a sequence $J \subset \NN$ of uniform density $1$ such that for all $f \in C(X), \lim_{n \in J} T^n_\star \mu = m.$  (Moreover, this is never true when $\mu \notin P_1.$)

Weak twisting is also equivalent to the convergence of the uniform averages
\[ \lim_{N-M \to \infty} \frac{1}{N-M} \sum_{i=M}^{N-1} \left| \int f \circ T^i d \mu - \int f dm \right| = 0. \]
Furthermore, if one fixes $f,$ this convergence is uniform in $\mu.$ (Again, this is never true when $\mu \notin P_1.$)
\end{corollary}
\begin{proof}
It is a standard fact that for a uniquely ergodic system, the convergence of $N^{-1} \sum_{n=0}^{N-1} f(T^n x)$ is uniform in $x$ (see for instance \cite{walters}.)  In particular replacing $x$ by $T^{-M} x$ does not change the rate of convergence.  Applying this to the function $F$ defined in the proof of Proposition \ref{twistingProp} gives uniform convergence of the uniform averages.  The second equivalence follows from Lemma \ref{cesaroLem}.

When $\mu \notin P_1$ we follow the argument given at the beginning of this section regarding choices for $P'.$  Specifically, let $\nu$ be the projection of $\mu$ onto $K=K(X,T).$  The norm $\| \cdot \|_\star$ (which will be defined at the beginning of the next section) induces the weak$^\star$ topology on $P(K)$ and, under this norm, $T_\star$ acts by isometries.  Therefore, $\nu \neq m_K,$ $T^n_\star \nu$ is bounded away from $m_K$ (Haar measure on $K$) independently of $n.$  It follows that $T^n_\star \mu$ is bounded away from $m.$
\end{proof}

\begin{example}
Minimal equicontinuous systems are twisting because $P_1 = \{ m \}.$
\end{example}

\begin{example} \label{uebutnotwpeEx1}
There are systems which are uniquely ergodic but not weakly twisting.  For instance, choose a (non-trivial) weakly mixing measure preserving system and use the Jewett-Krieger theorem (see, for instance, \cite{bellowfurstenberg})  to construct a uniquely ergodic topological realization $(X,T).$  Since
$(X,\mB, \mu, T)$ has no measure theoretic Kronecker factor, its maximal equicontinuous factor must be the one point system.  This tells us that $P_1(X) = P(X),$ which cannot be uniquely ergodic unless $X$ is one point (as discussed above.)

This example is not terribly satisfying.  Its construction relies on the (opaque) Jewett-Krieger theorem which produces systems on totally disconnected spaces.  See Question \ref{contskewQ}.
\end{example}

\begin{example}
So far, the explicit examples we have seen of invariant measures on $P_1(X)$ have all been of the form $\delta_\mu$ where $\mu$ is some invariant measure on $X.$  Given some collection $M$ of invariant measures and some measure $\theta$ on $M$ we can always define $\eta = \int_M \delta_\mu d \theta(\mu)$ and obtain another invariant measure on $P_1.$  These are the trivial examples.  For an invariant measure on $P_1$ which does not arise as a convex combination of point-masses at fixed points see Example \ref{nontrivialMeasureOnPONE}.
\end{example}

\begin{example}
Now we give an example of a system which is weakly twisting but not twisting.  Choose a subset $A \subset \ZZ$ having uniform density $1$ and an infinite complement.  Identify $A$ with a point $a \in \{0,1 \}^\ZZ.$  Write $\sigma$ for the shift $( \sigma x(n) = x(n+1) )$ and let $X$ be the closure of $\{ \sigma^n(a) : n \in \ZZ \}.$  Notice $X$ contains the shift invariant point $( \dots, 1,1,1, \dots) =: {\bf 1 },$ and therefore admits the invariant measure $\delta_{\bf 1}.$ It follows that the maximal equicontinuous factor of $(X,\sigma)$ is the one point system.  So, $P_1(X, \sigma) = P(X).$

The sets corresponding to the points of $X$ all have uniform density $1$ and that density is `achieved uniformly'.  That is, for any $\e>0$ there exists a $N$ such that for all $x \in X$, if $b-a > N$ then
\[ \frac{1}{b-a+1} \sum_{i=b}^a x(i) > 1- \e. \]
Fix $M>0$ and let $U = \{ x \in X: x(i) = 1, -M < i < M \}$ be a small neighborhood of ${\bf 1}.$  It follows from the work above, that for all $\e>0$ there exists $N'$ such that for any $x \in X$ and any $b-a > N',$
\[ \frac{1}{b-a+1} \sum_{i=b}^a \chi_U( \sigma^i x) > 1- \e. \]

Let $\mu \in P_1 = P(X).$  Then
\[ \frac{1}{b-a+1} \sum_{i=b}^a \sigma^i_\star \mu(U) =
\frac{1}{b-a+1} \sum_{i=b}^a \int_X \chi_U( \sigma^i x) d \mu(x) \geq
 1- \e. \]
 Since $\int_X \chi_U( \sigma^i x) d \mu(x) = \sigma^i_\star \mu (U) \leq 1,$ we get
 \[ \frac{1}{b-a+1} \sum_{i=b}^a \left| 1 -\sigma^i_\star \mu (U) \right| \leq \e. \]
But $\e$ was arbitrary and $U$ was an arbitrary cylindrical neighborhood of ${\bf 1}.$  So, if $f \in C(X)$ then
\[ \lim_{N \to \infty} \frac{1}{N} \sum_{i=1}^N \left| f({\bf 1}) - \int f d \sigma^i_\star \mu \right| = 0. \]
It follows from Proposition \ref{twistingProp} that $(X,\sigma)$ is weakly twisting and $\delta_{\delta_{\bf 1}}$ is the unique invariant measure on $P_1.$
However, $(X,T)$ is not twisting.  Indeed, choose $n_i \in \ZZ \sm A$ tending to infinity, let $f$ be the (continuous) characteristic function of  $\{ x \in X : x(0) = 1 \}$, and let $\mu = \delta_a.$  Then
\[ \int f d\sigma^{n_i}_\star \mu = f(\sigma^{n_i} a) = a(n_i) = 0 \neq 1 = f({\bf 1}). \]
\end{example}

\begin{question}
Do there exist topological systems admitting an invariant measure of full support, which are weakly twisting but not twisting?
\end{question}

This question, the related Question \ref{twistingRotationQ}, and Question \ref{lpunipotentcocycleQ} are, in the view of the author, the most interesting unresolved problems in this paper.

It is trivial (but pleasing) that $P_1$ is a functor.  To see this, suppose $\Phi:(X,T) \to (Y,S)$ is a factor map of topological systems and write $\pi_X, \pi_Y$ for the projections $X \to K(X,T), Y \to K(Y,S)$ respectively.
Then $\pi_Y \circ \Phi$ is a factor map.  By maximality of $K(X,T)$ we see this map must factor
as $\pi_Y \circ \Phi = \Phi_K \circ \pi_X$ for some unique map $\Phi_K:K(X,T) \to K(Y,S).$  In other words $K$ is a functor.  So
\[ (\pi_Y)_\star( \Phi_\star P_1(X,T) ) = (\Phi_K)_\star ( (\pi_X)_\star P_1(X,T)). \]
which proves $\Phi_\star( P_1(X,T) ) \subseteq P_1(Y,S).$  Write $\Phi_1 = \Phi_\star |_{P_1(X,T)}.$  That $\Phi \mapsto \Phi_1$ respects composition follows from the same statement about $\Phi \mapsto \Phi_\star.$

One subtle thing that should be verified is that $\Phi_1$ is surjective (since factor maps of topological systems are required to be surjective.)  Fix $\nu \in P_1(Y).$  Think of $C(Y),$ and $C(K(X,T))$ as subspaces of $C(X)$ and let $\nu'$ be the functional on $C(Y) + C(K(X,T))$ which agrees with $\nu$ on $C(Y)$ and Haar measure on $C(K(X,T)).$  Apply the Hahn-Banach Theorem to extend $\nu'$ to a linear functional $\mu$ on all of $C(X)$ satisfying $\mu(f) \leq \| f \|.$  Since $\mu(1) =1, \| \mu \| = 1.$  Write $\mu_+, \mu_-$ for the positive and negative parts of $\mu.$  Then
\[ \| \mu_+ \| - \| \mu_- \| = \mu(1) = 1 = \| \mu \| = \| \mu_+ \| + \| \mu_- \|. \]
So $\mu_- = 0,$ and $\mu$ is a probability measure.  Since $\mu$ agrees with Haar measure on $C(K(X,T)),$ it lies in $P_1(X,T).$

\begin{proposition}
Factors of twisting (weakly twisting) systems are also twisting (weakly twisting respectively.)
\end{proposition}
\begin{proof}
Notice that if a system has a fixed point which is a universal attractor, then the same can be said of every factor of that system.  Also notice that a factor of a uniquely ergodic system is uniquely ergodic.
The result now follows from the functoriality of $P_1$ and Proposition \ref{twistingdescriptionProp}.
\end{proof}

The next proposition is probably useless but the proof is too perversely entertaining to omit.

\begin{proposition}
If $(X,T)$ is weakly twisting then the maximal equicontinuous factor of $(P_1,T_\star)$ is trivial.
If $(X,T)$ is twisting then $(P_1,T_\star)$ is twisting.  If $(X,T)$ is weakly twisting but not twisting then the same holds for $(P_1,T_\star).$
\end{proposition}
\begin{proof}
By Corollary \ref{wktwistingCor}, for any $\mu \in P_1(X),$ $T^n_\star \mu \to m$ along some sequence.  The same is true for $T^n_\star (\pi \mu).$  Without loss of generality, $K(P_1)$ is an isometric system.  It follows that $\pi(\mu) = \pi(m).$  Therefore $K(P_1)$ is a single point.

By the preceding remark $K(P_1,T_\star)$ is minimal, which is necessary to make sense of $P_1(P_1,T_\star).$
We want to show that $P_1(P_1,T_\star)$ is uniquely ergodic.  Let $\theta$ be an invariant measure on $P_1(P_1,T_\star).$  This is an element of $P_1(P_1(P_1,T_\star), T_{\star \star})).$  The barrycenter $\theta'$ of $\theta$ is an element of $P(P_1)$ given by $\theta' := \int_{P_1(P_1,T_\star)} \eta d \theta(\eta).$
Invariance of $\theta$ gives us invariance of $\theta'.$  Therefore $\theta' = \delta_m.$  But the only way to take a convex combination of measures and get a $\delta$-measure is if the combination is degenerate.  In other words
$\theta = \delta_{\delta_m}.$  This proves $(P_1(P_1,T_\star), T_{\star \star})$ is uniquely ergodic.  Equivalently, $(P_1,T_\star)$ is weakly twisting.

If we additionally assume that $(X,T)$ is not twisting, then there is some measure $\mu \in P_1$ which does not attract to $m.$  It follows that $\delta_\mu$ does not attract to $\delta_m.$  So $P_1$ is not twisting.

Now assume $(X,T)$ is twisting.  Then it is also weakly twisting and once again, we can define $P_1(P_1,T_\star).$  Fix $\eta \in P_1(P_1,T_\star).$  For any $\e>0$ and for any neighborhood $U$ of $m$ in $P_1$ there exists $N$ such that for all $n>N$ we have $T^n_{\star \star} \eta ( U ) > 1-\e.$  It follows that any weak$^\star$ limit $\theta$ of $T^n_{\star \star} \eta$ satisfies $\theta( \{ m \} ) = 1.$  In other words, $\theta = {\delta_m}.$  This proves that $\delta_m$ is the unique attracting fixed point in $(P_1(P_1,T_\star),T_{\star \star}).$  In other words, $(P_1,T_\star)$ is twisting.

\end{proof} 


\section{Minimal rotations on nilmanifolds} \label{nilrotations}

In this section we derive Theorems \ref{nilThm} and \ref{nil2Thm} as easy corollaries of Theorem \ref{nilimprovementThm}.

Let $G$ be a connected, simply connected, nilpotent Lie group with Lie algebra $\g.$  Let $\G$ be a lattice in $G.$
It is well known that the exponential map provides a homeomorphism between $G$ and $\g$ (see \cite{malcev}.)  Write $G_i$ for the lower central series of $G$ and let $l$ be minimum with $G_l = 0.$  Write $X_i = G/ G_i \G.$  
Fix $u \in G$ such that $T(g\G) = ug \G$ is a minimal rotation on $X.$

\begin{theorem} \label{nilimprovementThm}
Let $\mu$ be a probability measure on $X$ which projects to Haar measure on $X_{l-1}.$  Then $T^n_\star \mu$ converges in the weak$^\star$ topology to Haar measure.
\end{theorem}

We will give two proofs of this theorem.  The second proof is shorter and relies on the fact that rotations on nilmanifolds have countable Lebesgue spectrum on the orthocomplement of $L^2(X_1)$ (see Green's article: chapter 5 in \cite{auslandergreenhahn}.)  The first proof is geometric in nature and shows how shearing causes invariance.  The first proof is longer, but it has two advantages: it stands alone and, more importantly, the method can be adapted to other situations.  In particular, the geometric method of the first proof applies to systems with far less algebraic structure, like the skew products discussed in section 4.

Before proving the theorem, we define a metric inducing the weak$^\star$ topology which makes some observations easier.  Strictly speaking, this is unnecessary.  However, it allows us to avoid explicitly mentioning test functions.  The author finds this notation convenient and hopes the reader will as well.

Given $\mu \in C(X)^\star$ define the norm.
\[ \| \mu \|_\star := \sup \{ \int f d \mu : f \in C(X), |f| \leq 1, f \text{  1-Lipschitz } \}. \]
The reader should be able to easily verify the triangle inequality.  This turns $C(X)^\star$ into a (usually) not complete normed linear space with the following nice property: let $B \subset C(X)^\star$ be bounded in operator norm.  Then the topology induced on $B$ by $\| \cdot \|_\star$ is the weak$^\star$ topology.  As we will see, simple geometric properties of maps on $X$ often translate into equally simple properties of the induced maps on $C(X)^\star.$

To prove that $\| \cdot \|_\star$ induces the correct topology, recall that $B$ is metrizable in the weak$^\star$ topology (really all the work is hidden in this fact.)  So it suffices to prove that a sequence converges under one topology if and only if it converges under the other.  Suppose $\mu_n$ is a sequence in $B.$  If $\mu_n$ converges to $\mu$ in the weak$^\star$ topology, then $\int f d\mu_n$ converges uniformly to $\int f d \mu$ for all $f$ in a compact subset of $C(X).$  Here we are using the fact that $f \mapsto \int f d\nu$ is itself $1$-Lipschitz.  But the set of all $1$-Lipschitz $f \in C(X), |f| \leq 1$  is compact.  So this tells us that $\| \mu_n - \mu \|_\star \to 0.$  Conversely, suppose $\| \mu_n - \mu \|_\star \to 0.$  Then for any Lipschitz function $f$ we have $| \int f d\mu_n - \int f d\mu | \to 0.$  But the Lipschitz functions are uniformly dense in the continuous functions.  So the same holds for any $f \in C(X).$

One obvious fact about $\| \cdot \|_\star$ is that isometries of $X$ induces isometries on $(C(X)^\star, \| \cdot \|_\star).$  Another observation we will need is that if $\Phi:X \to X$ moves each point by at most $\e,$ then 
$\| \Phi_\star( \mu ) - \mu \|_\star \leq \e \| \mu \|$ (where $\| \mu \|$ denotes the usual operator norm on $C(X)^\star.$)
This metric also has nice properties with respect to convex combinations : if we write $\mu = \int \mu_x d \nu(x)$ then
\[ \| \mu \|_\star \leq \int \| \mu_x \|_\star d \nu(x). \]
Suppose $X = \coprod_i C_i$ is a partition of $X$ into sets of positive measure.  Then we can define the conditional measures $\mu_x = \mu(C_i)^{-1} \mu|_{C_i}$ when $x \in C_i.$  Letting $\nu = \mu$ and applying the principle above yields
\[ \| \mu \|_\star \leq \sum_i \mu(C_i) \| \mu(C_i)^{-1} \mu|_{C_i} \|_\star \leq \sum_i \| \mu|_{C_i} \|_\star. \]

We shall also need to define some metrics on groups and their quotients.  Let $d_G$ be a right-invariant metric on $G.$  One can construct this by choosing an inner product on the Lie algebra of $G$ and then transporting this via right-translation to a Riemannian metric on $G.$  Then one defines $d_G(g,h)$ in the usual way by taking the infemum of lengths of differentiable curves connecting $g$ to $h.$  This allows us to define a metric on $X$ by
\[ d_X(a\Gamma, b\Gamma) = \inf_{\gamma \in \Gamma} d_G(a \gamma, b). \]
This metric has the property that $d_X(a\Gamma, b\Gamma)$ is equal to the size of the smallest $g$ such that $ga\Gamma = b\Gamma,$ where by ``size" we mean $d_G(g,1).$

Finally, recall the following results of Malcev \cite{malcev}.   We can choose a {\it canonical basis} for $\G.$  This is a collection $\gamma_1, \dots, \gamma_k \in \G$ such that
\begin{enumerate}
\item every element of $G$ can be written uniquely as a product $\gamma_1^{e_1} \cdots \gamma_k^{e_k}$ (where the $e_k$ are real numbers.)
\item for each $i$ the set of all elements of the form $\gamma_i^{e_i} \cdots \gamma_k^{e_k}$ is a normal subgroup $G_{(i)}$ of $G$
\item for each $i,$  $G_{(i)}/G_{(i+1)}$ is isomorphic to $\RR$
\item the sequence $G_{(i)}$ is a refinement of $G_j$  (that is, $G_j$ is a subsequence of $G_i.$)
\end{enumerate}

Both proofs of Theorem \ref{nilimprovementThm} begin the same way.

\ \\
{\it Proof of Theorem 3.1.}\ Choose a Malcev basis $\gamma_1, \gamma_2, \dots, \gamma_k$ as explained above and suppose $\gamma_{k'+1}, \dots, \gamma_k$ are those coordinates lying in $G_{l-1}.$  Let $F$ be the standard fundamental region for the action of $G$ on $X.$  That is,
\[ F = \{ \gamma_1^{e_1} \gamma_2^{e_2} \cdots \gamma_k^{e_k} : 0 \leq e_i < 1 \}. \]

Choose a finite partition of $[0,1)^{k-k'}$ into sets $C_1', C_2', \dots$  having
small diameter.  For each $i,$ let $C_i$ be the set of all $g = \gamma_1^{e_1} \gamma_2^{e_2} \cdots \gamma_k^{e_k} \in F$ with $(e_{k'+1}, \dots, e_k) \in C'_i.$  This is a partition of $F$ (or $X$ if we wish.)  Arbitrarily choose $c_i' = (f_{k'+1}, \dots, f_k) \in C_i'$ and define a map $\Phi_i: C_i \to C_i$ by
\[ \gamma_1^{e_1} \gamma_2^{e_2} \cdots \gamma_k^{e_k} \mapsto
\gamma_1^{e_1} \gamma_2^{e_2} \cdots \gamma_{k'}^{e_{k'}} \gamma_{k'+1}^{f_{k'+1}} \cdots \gamma_k^{f_k}. \]

Fix $\epsilon > 0.$  If the $C_i'$ are sufficiently small than each $C_i$ has diameter at most $\epsilon$ (in the right invariant metric on $G.$)  So each $\Phi_i$ moves points by at most $\epsilon.$   

Let $\mu_i$ be the pushforward of $\mu|_{C_i}$ under $\Phi_i.$  This is (probably) not a probability measure.  However $\sum_i \mu_i$ is a probability measure.  In fact, we will see (if we think of this measure on $X$) it closely approximates $\mu$ in a way that's invariant under multiplication by $u.$

Define $\Phi(g) = \Phi_i(g)$ when $g \in C_i.$  Then $\Phi$ moves each point by at most $\epsilon$ and only in the isometric direction.  So $T^n \Phi$ and $T^n$ are point-wise $\epsilon$-close.  Therefore
$ \| T^n_\star(\mu) - T^n_\star( \sum_i \mu_i) \|_\star  = \| T^n_\star(\mu) - T^n \Phi_\star(\mu) \|_\star \leq \epsilon.$

As explained above, if we project the canonical coordinates $\gamma_1, \gamma_2, \dots$ to $G/G_{l-1}$ then, except for $\gamma_{k'+1}, \dots, \gamma_k$ which vanish, we get canonical coordinates for the lattice $\Gamma G_{l-1}$ in $G/G_{l-1}.$  Let $Q_i$ be the image of $\Phi_i.$  Each $Q_i$ is a Euclidean cube in $F,$ which is bijectively mapped by $\pi$ to the standard fundamental domain $F_{l-1}$ in $G/G_{l-1}$ for $X_{l-1}$ with respect to the quotient coordinates.  Haar Measure $m_{l-1}$ on $X_{l-1}$ is the same as Lebesgue measure (with respect to canonical coordinates.) So, if we equip the cube $Q_i$ with Lebesgue measure $\lambda_i$ (of the appropriate dimension and with respect to canonical coordinates in $G,$) we see that $\pi |_{Q_i} : Q_i \to F_{l-1}$ is measure preserving.  Also notice that $\pi |_F: F \to F_{l-1}$ preserves $\mu$ by assumption and hence also preserves $\sum_i \mu_i$
(i.e. these measure project to Haar measure on $F_{l-1}.$)

Fix $j$ and let $N_j \subset Q_j$ be a $\lambda_j$-null set.  Let $N = \pi(N_j)$ and let $N_i = ( \pi|_{Q_i})^{-1}(N).$
Then
\[ 0 = \lambda_j(N_j) = m_{l-1}(N) = (\pi_\star \sum_i \mu_i)(N) = \sum_i \mu_i ( \pi^{-1} N)
= \sum_i \mu_i (N_i). \]
It follows that $\mu_j(N_j) = 0.$  We have proven that each $\mu_i$ is absolutely continuous with respect to $\lambda_i.$  The Radon-Nikodym Theorem allows us to write $d\mu_i = f'_i d \lambda_i$ where $f'_i$ is a measurable function on $Q_i.$  Note that $0 \leq f'_i \leq 1$ $\lambda_i$-almost everywhere.

The proof now continues in two ways.

\begin{proof}[First proof of Theorem \ref{nilimprovementThm}]
The commutators $[v,h]=z$ where $v \in G_{l-2}$ and $h \in G$ generate $[G_{l-2}, G]=G_{l-1}.$  We will show that when $n$ is large, $T^n_\star \mu$ is nearly invariant under the action of $z.$  By finite dimensionality, the same is then true for the action of any element of $G_{l-1}.$

Choose a continuous function $f_i$ on $Q_i, |f_i| \leq 1$ which agrees with $f_i'$ on a set $S_i \subseteq Q_i.$  Make this choice so that
$ \sum_i \mu_i (S_i) > 1- \epsilon. $  Then $\sum_i \mu_i$ and $\nu$ agree on the set $\bigcup_i S_i,$ and both give it measure at least $1-\epsilon.$

Write $d\nu = \sum_i f_i d \lambda_i.$  Then
\[ \| T^n_\star \mu - T^n_\star \nu \|_\star \leq  \| T^n_\star \mu - T^n_\star ( \sum_i \mu_i ) \|_\star
+ \| T^n_\star( \sum_i \mu_i - \sum_i f_i d \lambda_i ) \|_\star < 2 \epsilon. \]
 
We want to show that $T^n_\star \mu$ is nearly invariant under $z$ when $n$ is large.  The calculation above shows us that it suffices to show $T^n_\star \nu$ is nearly invariant under $z.$  The advantage of using $\nu$ instead of $\mu$ is that it is supported on geometrically nice pieces $Q_i$ of $F$ on which it is given by continuous density functions.  It is easier to understand what happens to such a measure when perturbed.

Write $u=\gamma_1^{e_1} \cdots \gamma_a^{e_a} g$ where $\gamma_1, \dots, \gamma_a$ are the elements of the Malcev basis of $\Gamma$ lying in $G \sm G_1$ and $g \in G_1.$
Let $v \in G_{l-2}$ be given by $v = \eta_1^{f_1} \cdots \eta_b^{f_b}$ where $\eta_1, \dots, \eta_b$ are the elements of the Malcev basis of $\Gamma$ lying in $G_{l-2}.$   Then
\[
[v, u^n] = [\eta_1^{f_1} \cdots \eta_b^{f_b}, (\gamma_1^{e_1} \cdots \gamma_a^{e_a} g)^n]
= \left( \prod_{i,j} [\eta_j, \gamma_i]^{f_j e_i} \right)^n.
\]
The $[\eta_j,\gamma_i]$ may not be linearly independent in $G_{l-1}.$  In fact, some must vanish.  But, since the one parameter subgroups through the $\gamma_1, \dots, \gamma_a$ generate $G$ and since the one parameter subgroups through the $\eta_j$ generate $G_{l-2},$ we know the one parameter subgroups through the $[\eta_j,\gamma_i]$ generate $G_{l-1}.$  In other words, the map
\[
\RR^{ba} \owns (t_{j,i}) \mapsto \prod_{i,j} [ \eta_j, \gamma_i]^{t_{j,i}} \in G_{l-1}
\]
is surjective.  It induces a map $\RR^{ba}/ \ZZ^{ba} \to G_{l-1} / \Gamma_{l-1}$ (where $\Gamma_{l-1}$ is the last nonzero term in the lower central series for $\Gamma.$)  Now let us choose the $f_j$ such that $\{ f_j e_i: j,i \}$ together with $1$ forms a collection of numbers which is linearly independent over $\QQ.$  To do this, choose $f_1$ to be transcendental over $1,e_1, \dots, e_a.$  Then choose $f_2$ to be transcendental over $1, e_1, \dots, e_a, f_1,$ etc.  Then, a $\QQ$-linear relationship between the $e_i f_j$ would allow us to solve for $f_j$ (with $j$ maximum) in terms of $1, e_1, \dots, e_a, f_1, \dots, f_{j-1},$ which would contradict the choice of $f_j.$

It follows that the backwards orbit $\{ n(f_j e_i)_{j,i} : n \in -\NN \}$ is dense in $\RR^{ba} / \ZZ^{ba}.$  So, it has dense image in $G_{l-1} / \Gamma_{l-1}$ which, in turn, surjects onto $G_{l-1}/ (\Gamma \cap G_{l-1})$ where it has dense image as well.  We have chosen an element $v \in G_{l-2}$ such that $\{ [v,u^n] : n < 0 \} $ has dense image in $G_{l-1}/(\Gamma \cap G_{l-1}).$

The commutator restricts to a continuous map $G_{l-2} \times G \to G_{l-1}$ which induces a map 
\[
G_{l-2}/G_{l-1} \times G/G_1 \to G_{l-1}
\]
of real vector spaces.  It follows from continuity and the identity $[xy,z] = [x,z]^y [x,y],$ that this map is bilinear.

Let $\lambda > 0$ and observe that $[v^\lambda, u^{\lambda^{-1} n}] = [v,u^n].$  This tells us that there exists some $N < 0$ such that for all $n<N$ we can choose $h \in G_{l-2}$ and $\gamma \in \Gamma \cap G_{l-1}$  such that $d_G(h,1) < \delta$ and  $d_G([h,u^n] \gamma, z) < \delta.$  We now use centrality of $\gamma$ and right invariance of the metric to get
\begin{align*}
d_G(zu^{-n}, u^{-n}h \gamma) &=  d_G(z, u^{-n}hu^n \gamma) \\
&\leq  d_G(z, [h,u^n]\gamma) + d_G([h,u^n]\gamma, u^{-n}hu^n \gamma) \\
&<  2 \delta.
\end{align*}

For any $g \in G$ we have
\[ 2\delta > d_X( zu^{-n} g \G, u^{-n} h \gamma g \G) = d_X( zu^{-n} g \G, u^{-n} h g \G). \]
In other words, the actions of $zu^{-n}$ and $u^{-n}h$ on $X$ are nearly the same.

To summarize, we have shown that, for all $\delta>0$ there exists a large integer $N'$ (equal to $-N$) such that for all $n>N'$ there is some $h \in G_{l-2}$ with $d_G(h, 1) < \delta$ and $d_X(zu^n.x, u^nh.x) < 2 \delta$ for all $x \in X.$

Now we study the push-forward of $\nu$ under $h.$  Given $g \in G$ we can use the canonical coordinates to write
\[ g = \gamma_1^{e_1} \gamma_2^{e_2} \cdots \gamma_{k'}^{e_{k'}} \zeta \text{ \ \ and \ }
hg = \gamma_1^{e'_1} \gamma_2^{e'_2} \cdots \gamma_{k'}^{e'_{k'}} \zeta' \]
where $\zeta, \zeta' \in G_{l-1}.$  Write $\kappa(g) = \zeta (\zeta')^{-1}.$  Write $\alpha$ for $\e$ divided by the total number of cells $C_i.$  Since $F$ is compact, by requiring $h$ be sufficiently close $1$ (i.e. by letting $\delta$ be sufficiently small in the argument above) we can achieve $d(\kappa(g), 1) < \alpha$ for all $g \in F.$  Also require $\delta < \alpha.$
Let $\nu'$ be the push-forward of $\nu$ under $g \mapsto hg,$ and let $\nu''$ be the push-forward of $\nu$ under $g \mapsto \kappa(g)hg.$  Since these two maps are point-wise $\alpha$-close and differ only in the isometric direction, as in the arguments above, we have $\| T^n_\star(\nu') - T^n_\star (\nu'') \|_\star < \alpha$ for all $n.$

Let $R_i := \{ \kappa(g)hg : g \in Q_i \} \cap Q_i.$  Both $h$ and $\kappa(g)$ are within $\alpha$ of $1$ by assumption, so $d(\kappa(g)hg, g) < 2 \alpha.$  Recall that $Q_i$ is a cube with respect to canonical coordinates that was created by fixing all coordinates lying in $G_{l-1}.$  Loosely speaking, $Q_i$ is a cube lying in a particular plane in $G.$
By the choice of $\kappa(g)$ we know that for $g \in Q_i, \kappa(g)hg$ lies in the same plane.  Since $Q_i$ has side length $1,$ it follows that $R_i$ contains a cube of the same dimension, centered in $Q_i,$ and having side length $1-4\alpha.$  Since $Q_i$ has dimension $k'$ we conclude $\lambda_i( R_i) \geq (1-4\alpha)^{k'}.$  Similarly
\[ \lambda_i( \{ \kappa(g)hg : g \in Q_i \} \cup Q_i ) \leq (1+4\alpha)^{k'} . \]
Therefore
\[ \lambda_i( \{ \kappa(g)hg : g \in Q_i \} \Delta Q_i ) \leq (1+4\alpha)^{k'} - (1-4\alpha)^{k'} =: \rho. \]

By construction, $\nu''|_{R_i}$ is absolutely continuous with respect to $\lambda_i$ and has density function $f''_i(g) := f_i(\kappa(g)hg).$  The functions $f_i$ form an equi-continuous family.  So, by requiring $\kappa(g)$ be sufficiently close $1$ we may assume $|f_i(g) - f''_i(g)|< \alpha$ for all $i$ and all $g \in R_i.$  We now have
\begin{align*}
\| T^n_\star \nu - T^n_\star \nu'' \|_\star
&= \| \sum_i (f_i \circ T^{-n}) dT^n_\star \lambda_i - \sum_i (f''_i \circ T^{-n}) dT^n_\star \lambda_i \|_\star \\
&\leq  \rho + \sum_i \| f_i \circ T^{-n} - f''_i \circ T^{-n} \| \lambda_i(R_i) \\
&= \rho + \sum_i \| f_i - f''_i \| \\
&\leq \rho + \sum_i \alpha = \rho + \e.
\end{align*}

Now we overestimate $\| T^n_\star \mu - z.T^n_\star \mu \|_\star$ by collecting results and applying the triangle inequality to the sequence of measures $T^n_\star \mu, T^n_\star \nu, T^n_\star \nu'', T^n_\star \nu' = z.T^n_\star \mu.$  This calculation yields $2 \e+ (\rho + \e) + \alpha.$ which tends to $0$ as $\e$ tends to $0.$

This implies that any weak$^\star$ limit $\theta$ of $T^n_\star \mu$ is invariant under $G_{l-1}.$  Write $\pi: G/\G \to G/G_{l-1} \G$ and let $\mA$ be the inverse image of the Borel $\sigma$-algebra on $X_{l-1}$ under $\pi$ (this is the algebra of $G_{l-1}$ invariant sets.) We then have that for $\theta$-almost every $x \in X$ the conditional measure $\theta_x^{\mA}$ is equal to Haar measure $\lambda_{\pi x}$ on the torus $\pi^{-1} \pi(x).$  Write $m_{l-1}$ for Haar measure on $X_{l-1}.$ By assumption $\pi_\star \mu = m_{l-1},$ so $\pi_\star \theta = m_{l-1}$ as well.  Therefore
\[
\theta = \int_X \theta_x^{\mA} d \theta(x)
= \int_{X_{l-1}}  \lambda_y d \pi_\star \theta(y)
= \int_{X_{l-1}}  \lambda_y d m_{l-1}(y)
= m.
\]
\end{proof}

\begin{proof}[Second proof of Theorem \ref{nilimprovementThm}]
Choose a continuous density function $\beta$ on $G_{l-1}$ such that $\beta \geq 0, \int_X \beta d m_{G_{l-1}} = 1,$ and $\beta \equiv 0$ outside an $\e$-neighborhood of $1.$  Define $f_i \in C(X)$ by $f_i(gz \Gamma) = f'_i(g) \beta(z)$ for $g \in C_i$ and $z$ in the $\e$-ball around $1$ in $G_{l-1}$ and let $f_i = 0$ elsewhere.  When $\e$ is smaller than the injectivity radius of $X$ this is well defined.  Notice that $f_i \in L^2(X).$

Since $d \nu_i := f_i dm$ can be perturbed in the central directions to yield $f'_i d \lambda_i,$ the two do not deviate from one another under the application of $T_\star.$  More formally, the map $zg \mapsto g$ which collapses $B^{G_{l-1}}_\e(1)C_i$ to $C_i$ takes $d\nu_i = f_i dm$ to $\mu_i = f'_i d \lambda_i$ and moves each point by at most $\e.$  Therefore,
\begin{align*}
\| T^n_\star \mu_i - T^n_\star \nu_i \|_\star &=  \| \mu_i - \nu_i \|_\star = \| f'_i d \lambda_i - f_i dm \|_\star \\
&\leq  \e \int f_i d \lambda_i, \text{ \ so \ } \\
\| T^n_\star (\sum_i \mu_i) - T^n_\star (\sum_i \nu_i) \|_\star &\leq  \e.  \text{ \ Finally, } \\
\| T^n_\star \mu - T^n_\star (\sum_i \nu_i) \|_\star &<  2\e.
\end{align*}

By construction $\nu := \sum_i \nu_i$ is absolutely continuous with respect to $m,$ so we can write $d \nu = f dm$.  Since $|f_i| < 1,$ $f$ is bounded, hence in $L^2(X).$  Being the density function of a probability measure, $\int f dm = 1.$  Let $\varphi \in C(X)$ be $1$-Lipschitz with $\sup | \varphi| \leq 1.$  Then
\begin{align*}
\lim_{n \to \infty} \int \varphi d T^n_\star \nu &= 
\lim_{n \to \infty} \int \varphi \circ T^n f dm
= \lim_{n \to \infty} \left< \varphi \circ T^n, f \right> \\
&=  \int \varphi dm \int f dm = \int \varphi dm.
\end{align*}
where the third equality comes from the fact that $(X,T)$ has countable Lebesgue spectrum on the orthocomplement of $L^2(X_1).$

Finally, let $L$ be any limit value of $\int \varphi d T^n_\star \mu.$  Then
\begin{align*}
\left| L - \int \varphi dm \right| &\leq 
\left| L - \lim_{n \to \infty} \int \varphi d T^n_\star \nu \right|
+ \left| \lim_{n \to \infty} \int \varphi d T^n_\star \nu - \int \varphi dm \right| \\
&<  \limsup_{n \to \infty} \| T^n_\star \mu - T^n_\star \nu \|_\star + 0 \\
&= 2 \e.
\end{align*}
Since $\e$ was arbitrary, It follows that $\lim_{n \to \infty} \int \varphi dT^n_\star \mu = \int \varphi dm.$  Since $\varphi\in C(X)$ was arbitrary, we have shown $\lim_{n \to \infty} T^n_\star \mu = m,$ as desired.
\end{proof}

\begin{corollary}
Let $\mu$ be a probability measure on $X$ whose projection onto
$X_{l-1}$ is absolutely continuous with respect to Haar measure.  Then any weak$^\star$ limit of $T^n_\star \mu$ is invariant under $G_{l-1}.$
\end{corollary}
\begin{proof}
This follows from the proof of Theorem \ref{nilimprovementThm}.
\end{proof}

\begin{corollary} \label{niltwistingCor}
Suppose $G$ is 2-step nilpotent group and $T$ is a minimal nilrotation on a compact nilmanifold $X=G/\Gamma.$
Then $(X,T)$ is twisting.
\end{corollary}
\begin{proof}
Apply Theorem \ref{nilimprovementThm}.
\end{proof}

\begin{corollary} \label{nilweaktwistingCor}
Minimal nilrotations are weakly twisting.
\end{corollary}
\begin{proof}
Let $P_i$ be the set of probability measures on $X$ which project to Haar measure on $X_i.$  Then
$P_1 \supset P_2 \supset \cdots \supset P_l = \{ m \}.$  By Theorem \ref{nilimprovementThm} we know that for $\mu \in P_i, i \geq 1,$ any subsequential limit of $T^n_\star \mu$ lies in the compact set $P_{i+1}.$  Equivalently $d(T^n_\star \mu, P_{i+1}) \to 0.$  Let $\theta$ be an invariant probability measure on $(P_1, T_\star).$  For any $\e>0$ there exists $n$ such that
\[ 1- \e < \theta(\{ \mu \in P_1 : d(T^n_\star \mu, P_2) < \e \} ) 
= \theta(\{ \mu \in P_1 : d(\mu, P_2) < \e \} ) \]
where the equality follows from invariance of $\theta$ and $P_2.$  We have 
\[ \theta(P_2) = \theta( \bigcap_{\e >0 } \{ \mu \in P_1 : d(\mu, P_2) < \e \} )
\geq \lim_{\e \to 0} 1-\e = 1. \]
Repeating this argument inductively yields $\theta(P_i) = 1$ for all $i$.  In particular $\theta(P_l) = \theta(\{ m \}) = 1.$  That is, $\theta = \delta_m.$
\end{proof}

\begin{proof} [Proof of Theorem \ref{nilThm}.]
Apply Corollary \ref{wktwistingCor} to the conclusion of Corollary \ref{nilweaktwistingCor}.
\end{proof}

\begin{proof} [Proof of Theorem \ref{nil2Thm}]
This is just a rephrasing of Corollary \ref{niltwistingCor}.
\end{proof}

\begin{question} \label{twistingRotationQ}
Can Corollary \ref{nilweaktwistingCor} be strengthened?  Are all minimal nilrotations twisting?
\end{question}


\section{Skew products} \label{skewproducts}

By a skew product we mean a system with space $X \times Y$ and transformation of the form $T(x,y) = (T_0(x), f(x,y)),$ where $T_0:X \to X$ is a homeomorphism and $f:X \times Y \to Y$ is continuous.  We say such a system has base $(X,T_0)$ and fiber $Y.$

Furstenberg proved in \cite{furstenberg} that a large class of skew products are uniquely ergodic.   He also discusses systems derived from an irrational circle rotation by repeatedly taking skew products.  In particular, he proves transformations on $\TT^d$ of the form
\[ T(x_1, \dots, x_d) = (x_1+ \alpha, x_2+f_1(x_1), \dots, x_d + f_{d-1}(x_1, \dots, x_{d-1})), \]
are uniquely ergodic when the $f_i$ are lipschitz and homotopically non-trivial.
We strengthen his result by showing that such systems are weakly-twisting, and when $d=2$ we show they are twisting.  

The goal of this section is to prove Theorems \ref{skewThm} and \ref{skew2Thm}.  Just as in the case of nilrotations, we will derive these theorems as corollaries of another theorem which tells us that when one takes weak$^\star$ limits of $T^n_\star \mu$ one gets measures with more invariance properties.

\begin{theorem}
Let $X=\TT^d$ and suppose $T:X \to X$ is of the form 
\[ T(x_1, \dots, x_d) = (x_1+ \alpha, x_2+f_1(x_1), \dots, x_d + f_{d-1}(x_1, \dots, x_{d-1})), \]
where $\alpha$ is irrational, $f_k$ is Lipschitz, and each $f_k(x_1, \dots, x_{k-1}, \cdot):\TT \to \TT$ is homotopically non-trivial.
If the projection of $\mu$ onto the first $d-1$ coordinates is absolutely continuous with respect to Haar measure on $\TT^{d-1}$ then any weak$^\star$ limit of $T^n_\star \mu$ is invariant under rotation in the last coordinate.
\label{skewimprovementThm}
\end{theorem}

Notice that the statement that $x_k \mapsto f_k(x_1, \dots, x_k)$ is homotopically non-trivial is independent of the choice of $x_1, \dots, x_{k-1}$ since all such choices lead to homotopic loops.

\begin{corollary} \label{weaktwistingskewCor}
Any Lipschitz iterated skew product system $(X,T)$ (as in Theorem \ref{skewimprovementThm}) is weakly twisting.
\end{corollary}
\begin{proof}
The proof is the same as that of Corollary \ref{nilweaktwistingCor}.  We stratify $P_1(X)$ into sets $P_i$ consisting of measures which, when projected onto the first $i$ coordinates, give Lebesgue measure.  By induction we use Theorem \ref{skewimprovementThm} to conclude that if $\theta$ is an invariant measure on $P_1(X)$ then $\theta(P_i) = 1$ for all $i.$  In particular $\theta(P_d) = \theta(\{ m \}) = 1,$ so $\theta = \delta_m.$
\end{proof}

\begin{corollary} \label{twistingskewCor}
Any Lipschitz skew product $(x,y) \mapsto (x+ \alpha, y+f(x))$ on $\TT^2$ with $\alpha$ irrational and $f$ homotopically non-trivial is twisting.
\end{corollary}
\begin{proof}
This follows immediately from Theorem \ref{skewimprovementThm} and the definition of twisting.
\end{proof}

As in the nilrotation case, Theorems \ref{skewThm} and \ref{skew2Thm} are trivial observations given Corollaries \ref{weaktwistingskewCor} and \ref{twistingskewCor} and the results of Section \ref{twisting}.  We leave this to the reader.

Before setting out to prove Theorem \ref{skewimprovementThm}, we first need an ergodic theorem for functions taking values in the upper unipotent group $U$ (Theorem \ref{unipotentcocycleThm}.)  For this we require lemmata.

\begin{lemma} \label{unipowerLem}
Suppose $u \in U$ and $\lambda = \lambda(j-i)$ is as in Theorem \ref{unipotentcocycleThm}.  Then
\[ \lim_{n \to \infty} \left| \theta_{1/n} (u^n)_{i,j} - \lambda u_{i,i+1} u_{i+1,i+2} \cdots u_{j-1,j} \right| = 0, \]
for all $i,j.$
\end{lemma}

\begin{proof}
We will prove by induction on $j-i$ that when $j-i \geq 0, (u^n)_{i,j}$ is a polynomial of degree at most $j-i$ in $n$ with the coefficient on $n^{j-i}$ equal to $\lambda u_{i,i+1} u_{i+1,i+2} \cdots u_{j-1,j}$ (from which the claim immediately follows.)  When $j-i = 0$ we have an empty product and the result is obvious.  For larger $j-i$
\begin{align*}
(u^n)_{i,j} &= (u u^{n-1})_{i,j} = \sum_{k=i}^j u_{i,k} (u^{n-1})_{k,j} \\
&= (u^{n-1})_{i,j} + \sum_{k=i+1}^j u_{i,k} (u^{n-1})_{k,j} \\
&= (u^{n-2})_{i,j} + \sum_{k=i+1}^j u_{i,k} (u^{n-2})_{k,j}  + \sum_{k=i+1}^j u_{i,k} (u^{n-1})_{k,j} \\
&\ \vdots  \\
&=  \sum_{m=0}^{n-1} \sum_{k=i+1}^j u_{i,k} (u^m)_{k,j} \\
&= \sum_{m=0}^{n-1} \left( u_{i,i+1} (u^m)_{i+1,j} + \sum_{k=i+2}^j u_{i,k} (u^m)_{k,j} \right).
\end{align*}
Write $P(m)$ for expression in parentheses. By our inductive hypothesis, the inner sum is a polynomial (in $m$) of degree strictly smaller than $j-i-1.$  Also by our inductive hypothesis, $u_{i,i+1} (u^m)_{i+1,j}$ is a polynomial of degree at most $j-i-1$ with coefficient on $n^{j-i-1}$ equal to $\lambda(j-i-1)u_{i,i+1} u_{i+1,i+2} \cdots u_{j-1,j}.$    So $P(m)$ inherits this property.  It follows that 
\[ (u^n)_{i,j} = \sum_{m=0}^{n-1} P(m) \]
is a polynomial of degree at most $j-i$ with coefficient on $n^{j-i}$ equal to $\lambda(j-i) u_{i,i+1} u_{i+1,i+2} \cdots u_{j-1,j}.$
\end{proof}

\begin{lemma} \label{unilimitLem}
For any $\e>0$ and $M>0$ there exists $\delta>0$ such that the following holds: Suppose
$\{ u,  u_n : n=1,2,\dots \}$ is a bounded subset of $U$ with all entries on the first super-diagonal bounded by $M.$ Also suppose that for any $n,$ all super-diagonal entries $(u_n)_{i,i+1}$ differ from $u_{i,i+1}$ by at most $\delta,$ then
\[ \limsup_{n \to \infty} \left|  \theta_{1/n}( u_1 u_2 \cdots u_n)_{i,j} - \theta_{1/n}(u^n)_{i,j} \right| < \e. \]
for all $i,j.$
\end{lemma}

\begin{proof}
By Lemma \ref{unipowerLem} we may as well assume $u_{i,j} = 0$ when $j-i \geq 2$ since this does not change the asymptotic behavior of the entries of $u^n.$  Notice that the conclusion we are trying to derive is invariant under conjugation by a diagonal matrix.  So we may conjugate everything by some diagonal matrix having only $1,-1$ on the diagonal to ensure that $u$ has non-negative entries on the first super-diagonal.  Let $M'/2$ be the universal upper bound on the entries of $u,u_n.$  Write $u_n = u + \delta_n$ where $\delta_n$ is strictly upper triangular, with $|(\delta_n)_{i,i+1}| < \delta$ and $|(\delta_n)_{i,j}| < M'.$  Now consider the matrix
\[ u_1 u_2 \cdots u_n - u^n= (u + \delta_1)(u + \delta_2) \cdots (u + \delta_n) - u^n. \]
Multiplying out the right side of this equation yields another instance of $u^n$ which cancels with the $-u^n$ to leave a sum of products of $u$ with the $\delta_i.$  The entries of $v$ can only increase if we replace $\delta_k$ by $| \delta_k |$ (where $| \delta_k |_{i,j} := | (\delta_k)_{i,j} |.$)  Let $\Delta$ be the strictly upper triangular matrix having $u_{i,i+1} + \delta$ on the first super-diagonal and $M'$ on all other super-diagonals.  This matrix has entries no smaller than the entries of any $| \delta_i |.$  Therefore
\[
(u_1 u_2 \cdots u_n - u^n)_{i,j}
\leq ((u+|\delta_1 |) \cdots (u + |\delta_n |) - u^n)_{i,j}
\leq ((u+\Delta)^n - u^n)_{i,j}.
\]
Applying Lemma \ref{unipowerLem} again yields
\[
\limsup_{n \to \infty} \frac{1}{n^{j-i} } (u_1 u_2 \cdots u_n - u^n)_{i,j} \leq \lambda \prod_{k=i}^{j-1} (u_{k,k+1} + \delta) - \lambda \prod_{k=i}^{j-1} u_{k,k+1}.
\]
A symmetric calculation gives the same estimate for the other difference $u^n-u_1 u_2 \cdots u_n.$  Therefore
\[
\limsup_{n \to \infty} \frac{1}{n^{j-i} } |(u_1 u_2 \cdots u_n - u^n)_{i,j}| \leq \lambda (M+\delta)^{d-1} - \lambda M^{d-1},
\]
which can be made as small as we like by requiring $\delta$ to be sufficiently small.
\end{proof}

In order to simplify notation, we use the langauge of {\it cocyles.}
By a cocycle on an invertible dynamical system taking values in a group $G,$ we mean a function $C:X \times \ZZ \to G$ with the property that $C(x,n+m) = C(T^m x, n)C(x,m).$  One can easily extend this definition by replacing $\ZZ$ with any group or semi-group.  With more complicated groups acting, the structure of a cocycle may be somewhat constrained.  With that integers acting (as in our definition), any map $f:X \to G$ yields a cocycle: when $n \geq 0$ we let
\[ ÊC(x,n) = f(T^{n-1}x) \cdots f(Tx) f(x) \text{ \ and \ } C(x,-n) = C(x,n)^{-1}. \]
One immediately sees that all cocycles arise in this way by taking $f(x) = C(x,1).$

\begin{proof}[Proof of Theorem \ref{unipotentcocycleThm}]
Assume $m$ is ergodic.   Unpacking the definition of $\theta_t,$ we see that the goal is to demonstrate the validity of the limit
\[
\lim_{n \to \infty} \frac{1}{n^{j-i}} C(x,n)_{i,j} = \lambda \prod_{k=i}^{j-1} \, \int_X f(x)_{k,k+1} dm.
\]

Fix $\e>0.$  Let $M = \sup_{i,y} | f(y)_{i,i+1} |,$ and let $\delta$ be given by Lemma \ref{unilimitLem}.
Choose $N$ sufficiently large that
\[ \left| \frac{1}{N} \sum_{n=0}^{N-1} f(T^n x)_{i,i+1} - \int_X f_{i,i+1} dm \right| < \delta. \]
for all $i$ and for all $x$ in a set $E$ with $m(E) > 1- \e^2.$

Notice that for any $x \in X,$
\[
\frac{1}{N} \sum_{n=0}^{N-1} f(T^n x)_{i,i+1}
= \frac{1}{N} C(x,N)_{i,i+1} = ( \theta_{1/N} C(x,N))_{i,i+1}
\]
With this in mind, let $v_k = v_k(x) = \theta_{1/N}C(T^{kN}x, N)$ and let $u \in U$ be the matrix with ones on the diagonal, $u_{i,i+1} = \int f_{i,i+1} dm,$ and zeros elsewhere.

Although we are assuming $(X,\mB, m, T)$ is ergodic, $(X, \mB, m, T^N)$ may not be.  Let $m = \int \mu_x dm(x)$ be the associated ergodic decomposition.  We also have the decomposition $m = \int N^{-1}( \mu_x + \mu_{Tx} + \cdots \mu_{T^{N-1}x}) dm(x). $  But
\[ T^\star( N^{-1}( \mu_x + \cdots \mu_{T^{N-1}x}) ) - N^{-1}( \mu_x + \cdots \mu_{T^{N-1}x}) = N^{-1}( \mu_{T^N x} - \mu_x), \]
which is zero $m$-almost everywhere.  We have written the $T$-ergodic measure $m$ as a convex combination of $T$-invariant measures, so the decomposition must be degenerate.  That is $m = N^{-1}( \mu_x + \mu_{Tx} + \cdots \mu_{T^{N-1}x})$ $m$-almost everywhere.  Let $x_0, x$ be two points for which this happens.  Any set which has positive measure for $\mu_x$ also has positive measure for $m = N^{-1}( \mu_{x_0} + \cdots \mu_{T^{N-1}x_0})$, and so must have positive measure for some $\mu_{T^l x_0}.$  But, distinct ergodic measures must be mutually singular.  So, it must be that $\mu_x = \mu_{T^l x_0}$ for some $l.$

Since $\e^2 > m(X \sm E) = \int \mu_x(X \sm E) dm(x),$ there is a set $Y \subset X$ with $m(Y) \geq 1-\e$ such that for $x \in Y, \mu_x(X \sm E) < \e$ (by the Chebychev inequality.)
By removing at most an $m$-null set from $Y$ we may assume that for every $x \in Y$ there exists $0 \leq l < N$ such that $\mu_x = \mu_{T^l x_0}$ (this follows from the previous paragraph.)  Fix $x \in Y.$
By the point-wise ergodic theorem, the set $S := \{ k : T^{kN} x \notin E \}$ has density $d(S) = \mu_{T^l x_0}(X \sm E) < \e.$  Therefore, when $K$ is sufficiently large, we know that $|S \cap [0,K) | < \e K. $

Define a new sequence $u_k$ by letting $u_k = v_k$ when $k \notin S$ (when $T^{kN}x \in E$) and let $u_k = u$ otherwise.
In the expression
\[ \theta_{1/N}C(x,KN) = v_0 v_1 \cdots v_{K-1}, \]
rewrite $v_k$ as $u - (u-v_k)$ for each $k \in S.$  Multiplying out the resulting expression yields $u_0 u_1 \cdots u_{K-1}$ together with many summands which we will study presently.

Each of these summands is a product of matrices and each matrix is either a $v_k$ or it is of the form $u-v_k.$
Notice that $u-v_k$ is a strictly upper triangular matrix.  Let $L = 2 \sup_{y,i,j} | \theta_{1/N} C(y,N)_{i,j}|.$  Let $A \in U$ be the matrix with $A_{i,j} = L$ when $i<j$ and let $B = A-1.$  Then
\[ |(v_k)_{i,j}| \leq A_{i,j} \text{ \ and \ } |(u-v_k)_{i,j}| \leq B_{i,j}. \]
So we can overestimate the size of the entries of a summand by replacing every $v_k$ by $A$ and every $u-v_k$ by $B.$  Conveniently, $A$ and $B$ commute and $B^d = 0.$  So
\begin{align*}
| (v_0 \cdots v_{K-1})_{i,j} - (u_0 \cdots u_{K-1})_{i,j}|
& \leq \sum_{n=1}^{d-1} \binom{|S|}{n} ( A^{K-n} B^n )_{i,j}.
\end{align*}
There is a constant $L'$ depending only on $L$ and $d$ such that for all $n,$ every entry of $B^n$ is bounded by $L'.$  So,
\[ ( A^{K-n} B^n )_{i,j} \leq \sum_{k = i}^{j-n} (A^{K-n})_{i,k} L'. \]
Now we apply Lemma \ref{unipowerLem} (the proof if not the statement) to see that $(A^{K-n})_{i,k}$ is bounded by a polynomial (in $K$) of degree at most $k-i$ with coefficients depending only on $f$ and $d.$  Adding these up, we see that the expression above is bounded by a polynomial $p_{i,j,n}(K)$ of degree at most $j-n-i$, again, with coefficients only depending on $f$ and $d$.  Finally, we have
\begin{align*}
| (v_0 \cdots v_{K-1})_{i,j} - (u_0 \cdots u_{K-1})_{i,j}|
& \leq \sum_{n=1}^{d-1} \binom{\e K}{n} p_{i,j,n}(K), 
\end{align*}
which is a polynomial of degree
\[ \max \{ \deg \binom{\e K}{n} + \deg p_{i,j,n}(K) : 1 \leq n \leq d-1 \} \leq n + (j-n-i) = j-i, \]
with $\e$ dividing the leading coefficient.  Write $\e p(K)$ for this polynomial.

  Notice that the matrix $u$ and the sequence $u_k$ satisfy the hypotheses of Lemma \ref{unilimitLem}.  Therefore
\begin{align*}
& \limsup_{K \to \infty} \left| \frac{1}{K^{j-i}} (u_0 u_1 \cdots u_{K-1})_{i,j} - \frac{1}{(KN)^{j-i}}(u^{KN})_{i,j} \right| \leq  \\
& \limsup_{K \to \infty} \frac{1}{K^{j-i} } \left| (u_0 u_1 \cdots u_{K-1})_{i,j} - (u^K)_{i,j} \right| 
+ \left| \frac{1}{K^{j-i}}(u^K)_{i,j} - \frac{1}{(KN)^{j-i}} (u^{KN})_{i,j} \right| \\
& = \limsup_{K \to \infty} \frac{1}{K^{j-i} } \left| (u_0 u_1 \cdots u_{K-1})_{i,j} - (u^K)_{i,j} \right| + 0  < \e,
\end{align*}
where the second summand in the middle of this calculation vanishes because, by Lemma \ref{unipowerLem}, it is a difference of two expressions having the same limit.

Let $M'$ be larger than the absolute value of any matrix entry of $C(y,l)$ or $u^l$ where $l$ ranges over $[0, N-1]$ and $y$ ranges over $X.$  
By Lemmas \ref{unipowerLem} and \ref{unilimitLem}, we can find a polynomial $q(n)$ of degree at most $j-i-1$ which gives an upper bound for
$|C(x,n)_{i',j'}|$ for all $n$ and all $j'-i' < j-i.$ and also for $|u^n_{i',j'}|.$  It now follows from simply writing out the multiplication
\begin{align*}
C(x, KN+l) = C(T^{KN}x, l) C(x, KN), \text{ \ \ that } \\
| C(x, KN+l)_{i,j} - C(x, KN)_{i,j}| < d q(KN)M'.
\end{align*}
The same argument gives us $|u^{KN+l}_{i,j} - u^{KN}_{i,j}| < d q(KN)M'.$
Finally, for $K$ sufficiently large, we have
\begin{align*}
&\left| C(x,KN+l)_{i,j} - (u^{KN+l})_{i,j} \right| \leq \\
&\left| C(x,KN+l)_{i,j} - C(x,KN)_{i,j} \right| 
+ \left| C(x,KN)_{i,j} - N^{j-i}(u_0 u_1 \cdots u_K)_{i,j} \right| \\
&+  \left| N^{j-i} (u_0 u_1 \cdots u_K)_{i,j} - (u^{KN})_{i,j} \right|
+ \left| (u^{KN})_{i,j} - (u^{KN+l})_{i,j} \right|  \\
&< d q(KN)M' + N^{j-i} \e p(K) + \e (KN)^{j-i} + d q(KN)M' \\
&= \e \left[ N^{j-i} p(K) + (KN)^{j-i} \right] +  2 d q(KN)M'.
\end{align*}
Recall that $\deg q < j-i$ and $\deg p \leq j-i.$  So, we have a polynomial in $K$ of degree at most $j-i$ with $\e N^{j-i}$ dividing the coefficient on $K^{j-i}.$  Therefore, if
we divide this whole expression by $(KN+l)^{j-i}$ and take a limit, we get
\begin{align*}
\limsup_{KN+l \to \infty} \frac{1}{(KN+l)^{j-i}} \left| C(x,KN+l)_{i,j} - (u^{KN+l})_{i,j} \right|  < c \e,
\end{align*}
where $c$ is some constant depending only on $f$ and $d.$  Now we apply Lemma \ref{unipowerLem} and conclude that
for all $x$ in a set $Y$ of measure at least $1-\e,$
\begin{align*}
\limsup_{n \to \infty} \left| \frac{1}{n^{j-i}} C(x,n)_{i,j} - \lambda \prod_{k=i}^{j-1} \, \int_X f(x)_{k,k+1} dm \right| < c \e.
\end{align*}
Since $\e>0$ was arbitrary, in fact, the limit is zero almost everywhere.  This completes the proof of the ergodic case of the Theorem.

If $m$ is not ergodic then let $m = \int m_x dm(x)$ be its ergodic decomposition.  Let $X'$ be the set of all $x \in X$ for which $\theta_{1/n}C(x,n)$ converges.  Regrettably, we must show this set is measurable.  Chose a countable dense set $\{w_k : k \in \NN \}$ in $U.$  Then the measurable set
\[ X'(k, \e) := \bigcup_{N=1}^\infty \bigcap_{n = N}^\infty \{ x \in X: \theta_{1/n}C(x, n) \in B_\e(w_k) \} \]
consists of all $x \in X$ for which $\theta_{1/n}C(x, n)$ eventually always lies in the $\e$-ball around $w_k.$  Now
\[ X' = \bigcap_{n=1}^\infty \bigcup_{k=0}^\infty X'(k, 1/n) \]
is measurable as claimed.

By the ergodic case, $m_x(X') = 1,$ so
\begin{align*}
m(X') &= \int m_x(X') dm(x) = \int 1 dm = 1,
\end{align*}

Finally, for almost every $x$, we have computed the exact value of $f^*(x)$ and it depends only on $m_x.$  Since $m_{Tx} = m_x$ for $m$-almost every $x,$ we have $f^* \circ T = f^*$ almost everywhere. 
\end{proof}

The proof of Theorem \ref{ueunipotentcocycleThm} is easier and similar to that of Theorem \ref{unipotentcocycleThm}.  So, we only sketch it here.

\begin{proof}[Proof of Theorem \ref{ueunipotentcocycleThm}]
Fix $\e>0.$  Let $M = \sup_{i,y} | \theta_{1/N} C(y,N)_{i,i+1} |,$ and choose $\delta$ as in Lemma \ref{unilimitLem}.
Choose $N$ sufficiently large that
\[ \left| \frac{1}{N} \sum_{n=0}^{N-1} C(T^n x, 1)_{i,i+1} - \int_X C(y,1)_{i,i+1} dm(y) \right| < \delta. \]
for all $i$ and for all $x \in X.$  Let $x$ be any point at all, and let $u_k = \theta_{1/N}C(T^k x, N).$  Let $u$ be as in the proof of Theorem \ref{unipotentcocycleThm}.  Now we can skip most of the difficulties and apply Lemma \ref{unilimitLem} to $u$ and $u_k.$  Use the same argument as in Theorem \ref{unipotentcocycleThm} to finish the proof.
\end{proof}

A short comment about coordinates:  In $\RR$ or in $\RR^d,$ there is only one reasonable way to dilate: one must multiply by a scalar.  But in $U$ this is not the case.  The one parameter family $\theta_t$ of dilations we used is in no way canonical.  For instance, we could define $\theta'_t(u) = v \theta_t( v^{-1}uv) v^{-1},$ for some $v \in U.$  In other words, we could conjugate $\theta_t$ by an inner automorphism of $U.$  This results in nothing but a change of coordinates.  So, obviously, Theorem \ref{unipotentcocycleThm} works just as well with $\theta'$ in the different coordinate system.  Assuming we have an ergodic system we get
\[ \lim_{n \to \infty} \theta'_{1/n} ( f ( T^{n-1} x) \cdots f(T^2 x) f (Tx) f(x) )_{i,j} = \lambda \prod_{k=i}^{j-1} \, \int_X (v^{-1} f v)_{k,k+1} dm. \]

\begin{proof}[Proof of Theorem \ref{skewimprovementThm}.]
The strategy of this proof is similar to the geometric proof of Theorem \ref{nilimprovementThm}.  However, in place of the full $G$ action on $X,$ which is available to us in that Theorem, we employ Theorem \ref{unipotentcocycleThm}.

Suppose $\mu$ projects to Haar measure on $\TT^{d-1}.$  Fix $\e >0$ and define a partition $\{ C_i \}_i$ of $\TT^d \sim [0,1)^d$ by $(x_1, \dots, x_d) \in C_i$ if $i\e \leq x_d < (i+1)\e.$  Let $Q_i = \{ (x_1, \dots, x_{d-1}, i\e) : x_i \in \TT \} \subset C_i$ and let $\Phi_i:C_i \to Q_i$ be the projection given by $\Phi_i( x_1, \dots, x_{d-1}, x_d) = (x_1, \dots, x_{d-1}, i\e).$  Write $\mu_i$ for the pushforward of $\mu|_{C_i}$ under $\Phi_i$ and let $\mu' = \sum_i \mu_i.$  

Notice that $\mu$ and $\mu'$ have the same projection onto the first $(d-1)$-coordinates.  Assemble the maps $\Phi_i$ into one map $\Phi$ by setting $\Phi|_{C_i} = \Phi_i.$  Then $\mu' = \Phi_\star \mu.$  Since $T^n \Phi$ and $T^n$ are point-wise $\e$-close,
$ \| T^n_\star \mu - T^n_\star \mu' \|_\star < \e.$

Write $\lambda$ for Lebesgue measure on $\TT^{d-1}$ and $\lambda_i$ for Lebesgue measure on the $(d-1)$-dimensional tori $Q_i.$  Fix $j$ and suppose $N_j \subset Q_j$ is a $\lambda_j$-null set.  Let $N = \pi(N_j)$ and let $N_i = ( \pi|_{Q_i})^{-1}(N)$ where $\pi$ is projection onto the first $(d-1)$-coordinates.  Since $\pi_\star(\mu)$ is absolutely continuous with respect to Lebesgue measure, and since $\lambda(N) = 0,$ we must have
\[
0 = \pi_\star \mu (N) = (\pi_\star \sum_i \mu_i)(N) = \sum_i \mu_i ( \pi^{-1} N)
= \sum_i \mu_i (N_i).
\]
It follows that $\mu_j(N_j) = 0.$  We have proven that each $\mu_i$ is absolutely continuous with respect to $\lambda_i.$  This allows us to write $d\mu_i = \varphi'_i d \lambda_i$ where $\varphi'_i$ is a measurable function on $Q_i.$
Choose a continuous function $\varphi_i$ on $Q_i$ which agrees with $\varphi_i'$ on the set $S_i \subseteq Q_i$ and write $d\nu = \sum_i \varphi_i d \lambda_i.$  Make this choice so that $0 \leq \varphi_i \leq 1$ and
$ \sum_i \lambda_i (S_i) > 1- \e. $  It follows that for all $n$
\[
\| T^n_\star \mu - T^n_\star \nu \|_\star
\leq \| T^n_\star \mu - T^n_\star \mu' \|_\star + \| T^n_\star \mu' - T^n_\star \nu \|_\star  < 2\e.
\]
Choose $\e'>0$ such that $d(x,y) < \e'$ implies $|\varphi_i(x) - \varphi_i(y)| < \e^2$ for all $i.$

Since the skewing maps $f_k$ are Lipschitz, so is $T.$  By Rademacher's theorem the derivative $D_xT$ (the Jacobian of $T$ at $x$) exists almost everywhere.
The structure of the map $T$ tells us that, when it exists, $D_xT$  lies in $U$ (we need to use the ordered basis $(\partial/\partial x_d, \dots, \partial/\partial x_1).$)  Therefore $x \mapsto D_x T$ is a bounded measurable function and we can apply Theorem \ref{unipotentcocycleThm} to conclude $\theta_{1/n} (D_x T^n) = \theta_{1/n}( D_{T^{n-1}x} T \circ \cdots \circ D_xT)$ converges almost everywhere to some matrix $u \in U$ with
\[ u_{i,j} = \lambda \prod_{k=i}^{j-1} \, \int_X (D_x T)_{k,k+1} dm \text{ \ for \ } j \geq i \geq 1 \text{ \ and \ } \lambda = \lambda(i,j) > 0. \]
The assumption  on the loops $\gamma_k(x_k) := f_k(x_0,x_1, \dots, x_k)$ implies that for each $k$ the integral of $(D_xT)_{k,k+1}$ with respect to Lebesgue measure on $\TT^d$ is a nonzero integer.  Indeed, each $\gamma_k$ is a map $\TT \to \TT$ and so represents a class $[ \gamma_k ] \in \pi_1(\TT)$ which is isomorphic to $\ZZ$ with isomorphism given by $[ \gamma ] \mapsto \int \gamma' dm_\TT$ (if we think if $\gamma'$ as taking values in $\RR.$)  Since the homotopy class of $\gamma_k$ is independent of the choice of $x_0, \dots, x_{k-1}$ we see that $\int (D_xT)_{k,k+1} dm$ is again this nonzero integer.  

We have shown that for every $j \geq i, D_x(T^n)_{i,j} /n^{j-i}$ tends uniformly to some non-zero number.  Assume $\e$ was chosen sufficiently small that each of these non-zero numbers is greater than $\e$ in absolutely value.  Fix $t>0.$  We will show that when $n$ is large $T^n_\star \nu$ is nearly invariant under rotation by $t$ in the last coordinate.
Write $R(x_1, \dots, x_d) = (x_1, \dots, x_d + t)$ and $S(x_1, \dots, x_d) = S(x_1+ \delta, \dots, x_d)$ where $\delta$ will be determined later.  Assume $\delta < \e'.$  Then
\begin{align*}
\| T^nS_\star \nu - T^n_\star \nu \|_\star
&\leq  \sum_i \| T^n_\star( |\varphi_i \circ S - \varphi_i |d \lambda_i ) \|_\star \\
&\leq  \sum_i \| |\varphi_i \circ S - \varphi_i |d \lambda_i  \| \\
&\leq  \sum_i \e^2 < 2 \e
\end{align*}
(since there are approximately $1/\e$ indices $i.$)

Fix $(x_1, \dots, x_d)$ and let $\gamma(s) = (x_1(1-s) + (x_1+\delta)s, x_2, \dots, x_d)$ be the linear curve connecting $(x_1, \dots, x_d)$ to $S(x_1, \dots, x_d).$  Then
\[
T^n S(x_1, \dots, x_d)  = T^n(x_1, \dots, x_d) + \int_0^1 D_{\gamma(s)} T^n
\left( \begin{array}{c} 0 \\ \vdots \\ 0 \\ \delta \end{array} \right) ds.
\]
Let $\delta = \delta(n) = t/( u_{1,d} n^{d-1}).$  Then
\[
\lim_{n \to \infty} D_{\gamma(s)} T^n \left( \begin{array}{c} 0 \\ \vdots \\ 0 \\ \delta \end{array} \right)
= \lim_{n \to \infty} t \left( \begin{array}{c}
(u_{1,d} n^{d-1})^{-1} (D_{\gamma(s)} T^n)_{1,d} \\
\vdots \\
(u_{1,d} n^{d-1})^{-1} (D_{\gamma(s)} T^n)_{d,d}
\end{array} \right)
= \left( \begin{array}{c} t \\ 0 \\ \vdots \\ 0 \end{array} \right)
\]
So, by taking $n$ sufficiently large and $\delta$ as above, we can ensure that the maps $T^n S$ and $R T^n$ are point-wise $\e$-close.  This implies that
$\| RT^n_\star \nu - T^nS_\star \nu \|_\star < \e. $  Combining results yields
\[
\| RT^n_\star \nu - T^n_\star \nu \|_\star \leq \| RT^n_\star \nu - T^nS_\star \nu \|_\star
+ \| T^nS_\star \nu - T^n_\star \nu \|_\star < 3\e.
\]
Combining yet more results: $\| RT^n_\star \mu - T^n_\star \mu \|_\star $
\begin{align*}
&\leq   \| RT^n_\star \mu - RT^n_\star \nu \|_\star
+ \| RT^n_\star \nu - T^n_\star \nu \|_\star + \| T^n_\star \nu - T^n_\star \mu \|  \\
&<   2\e + 3\e + 2\e.
\end{align*}
This completes the proof.
\end{proof}

\begin{question} \label{contskewQ}
Are there skew products on $\TT^2$ of the form $T(x,y) = (x + \alpha, y + f(x))$ which are uniquely ergodic but not weakly twisting, or weakly twisting but not twisting?  Necessarily, $f$ must be non-Lipschitz.
\end{question}

\begin{question} \label{lpunipotentcocycleQ}
In Theorem \ref{unipotentcocycleThm} we take $f:X \to U$ to be bounded and measurable.  Instead, let us assume that its coordinate functions $f_{i,j}$ lie in $L^p(X).$
For what values of $p$ does the theorem still hold?
\end{question}


\section{Skew products with expansive fibers}

In the proofs of Theorems \ref{nilimprovementThm} and \ref{skewimprovementThm}, we repeatedly exploit the fact that our system is an isometric extension of a simpler system.  This allows us to perturb our measures in the isometric direction to yield a new measure which does not deviate from the original measure under application of $T.$  Not all skew products are of this convenient form.  Even for skew products $(x,y) \mapsto (x + \alpha, 2y+f(x))$ (on the two dimensional torus) we see exponential growth in the $y$ direction.  In this case it is not reasonable to expect that such strong equidistribution results hold.

We will now show it is possible that even Lebesgue $\mu$ measure supported on a horizontal line $x \mapsto (x,y_0)$ may fail to equidistribute.  Even worse, for some choices of $f$, we will see that repeated application of $T$ yields measures that are never close to Haar measure.  The next theorem gives a very precise picture of what weak$^\star$ limits $\theta$ of $T^n_\star \mu$ look like in this case.  It decomposes $\TT^2$ into two sets.  On one set, $\theta$ restricts to Lebesgue measure and on the other $\theta$ is supported on a Lipschitz curve.

\begin{theorem} \label{pskewThm}
Let $X=\TT^2$ and let $T(x,y) = (x+ \alpha, p y+f(x))$ where $\alpha$ is irrational, $|p| \geq 2$ is an integer, and $f:\TT \to \TT$ is continuously differentiable.  
Let $\mu$ be a probability measure on $X,$ supported an the graph of a differentiable curve $x \mapsto (x, \gamma(x))$ whose projection $\pi_\star \mu$ onto the first coordinate is absolutely continuous with respect to Haar measure, and let $\theta$ be any weak$^\star$ limit of $T^n_\star \mu.$  Let
\[ S = \{ x \in \TT : p \gamma'(x) + \sum_{n=0}^\infty \frac{1}{p^n} f'(x+n\alpha) \neq 0 \}. \]
Then there exists some $\beta \in \TT$ such that $\theta |_{\pi^{-1}(\beta + S)}$ is invariant under vertical rotation.  In particular, if $\mu$ projects to Lebesgue measure on the ciricle, then $\theta |_{\pi^{-1}(\beta + S)}$ is Lebesgue. Furthermore, on each connected component of $\pi^{-1}(\TT \sm \beta +S), \theta$ is supported on the graph of a Lipschitz curve with Lipschitz constant $\frac{p}{p-1} \sup | f' |.$
\end{theorem}

Notice that this theorem does not require that $f$ be homotopically nontrivial.

\begin{proof}
Write $\tau(x) := p\gamma'(x) + \sum_{n=0}^\infty p^{-n} f'(x+n\alpha).$
First observe that
\[ T^{n}(x,\gamma(x)) = (x+n \alpha,  p^n \gamma (x) + \sum_{k=0}^{n-1} p^{n-1-k}  f(x+k \alpha) ). \]
If we take the derivative of the $y$ coordinate we get
\[ \Delta_n(x) := p^n \gamma'(x) + \sum_{k=0}^{n-1} p^{n-1-k}  f'(x+k \alpha). \]
Therefore $|\Delta_n(x) - p^{n-1} \tau(x)| \leq \frac{p}{p-1} \sup | f' |.$   This proves that $|\Delta_n(x)| \to \infty$ when $x \in S$ and is bounded by $\frac{p}{p-1} \sup | f' |$ otherwise.  Write $\kappa = \frac{p}{p-1}\| f' \|.$  
In particular, $T^n(x, \gamma(x))$ is $\kappa$-Lipschitz on $\TT \sm S.$

Let $\theta$ be a subsequential limit of $T^n_\star \mu$ and pass to a further subsequence along which $n \alpha$ converges to some point $\beta \in \TT$ and $T^n(x, \gamma(x))$ converges point-wise to some function $x \mapsto (x+\beta, \eta(x)).$  From now on $n$ will always represent an element of this subsequence.   Point-wise convergence of functions on a compact set, all satisfying the same Lipschitz condition, is necessarily uniform convergence.  This proves that $\theta |_{\pi^{-1}(  \TT \sm \beta+S )}$ is supported on the curve $(x+ \beta, \eta(x))|_{\TT \sm S}.$  Furthermore, the limiting curve $\eta$ is necessarily $\kappa$-Lipschitz on the connected components of $\TT \sm S.$ 

It remains to be shown that $\theta|_{\pi^{-1}(\beta + S)}$ is invariant under vertical rotation.
To make things simpler,  instead of $T^n_\star \mu,$ we will study the push-forward $\nu_n$ of Lebesgue measure on $S$ by the map $x \mapsto (x, \gamma_n(x)),$ where
\[
\gamma_n(x) := p^n \gamma (x-n\alpha) + \sum_{k=0}^{n-1} p^{n-1-k}  f(x+k \alpha - n\alpha)
\]
The only difference between $T^n_\star \mu$ and $\nu_n$ is a horizontal translation. 
  
We assumed that $\mu$ projects to some measure absolutely continuous with respect to Lebesgue measure on $\TT.$  It suffices to consider the case where $\mu$ projects to Lebesgue measure, since, if $\rho$ is a continuous, vertically invariant function on $\TT^2, \lim  \rho d \nu_n = \rho \lim d\nu_n,$ and by Lusin's Theorem, the density function for $\mu$ is continuous on a set of as large measure as we wish.

Fix $\e>0.$ Let $S' = \{ x : | \tau(x) | > \e \}$ and assume $\e$ is small enough that the Lebesgue measure of $S'$ is within $\e$ of the Lebesgue measure of $S.$ Choose $\delta>0, \delta < \e$ such that  for any $x_1, x_2 \in S'$ with $d(x_1,x_2) < \delta$ we have $|\tau(x_1) - \tau(x_2) | < \e^2.$  Choose $N$ such that $p^N \e \delta > 2,$ and also such that
\[
\frac{(p^{N-1} (\e + \e^2) + \kappa)}{(p^{N-1}(\e - \e^2) - \kappa)}
\]
and it's reciprocal are both within $3\e$ of $1$ (this is possible when $\e < 1/3.$)

Let $n>N.$
We will construct a partition of $S'$ into intervals $[a_i, b_i)$ together with an exceptional set $E.$  Since $S'$ is open, it is union of disjoint open intervals.  Ignore all intervals having length less than $\delta.$  Each remaining interval we write as a union of as many contiguous intervals $[a,b)$ as possible with the property that $d(a,b) < \delta$ and such that, on each $[a,b), \gamma_n$
is monotonic, and $\gamma_n(a) = \gamma_n(b) = 1$ (this is possible because of the first assumption on the size of $N.$)  Furthermore, assume $\gamma_n(c) \neq 1$ for all $a < c < b.$  Write $\{ [a_i, b_i) : i \in I \}$ for this collection of intervals.  Let $E=S - \bigcup_i [a_i, b_i).$  If $\delta$ is sufficiently small than $\mu(E) < \e.$  Put more succinctly, we break $S'$ into intervals on which $\gamma_n$ wraps once, monotonically around $\TT$ together with an exceptional set $E$ consisting of `scraps'.

 Consider any rectangle of the form $R := [a_i,b_i) \times [y_1, y_2].$  This rectangle is intersected by the curve $(x, \gamma_n(x))$ in exactly one arc.  To avoid unnecessary cases, let's assume that $\gamma_n$ is increasing on $[a_i,b_i)$ (if $\gamma_n$ is decreasing the argument is analogous.)  Notice that $\gamma_n|_{[a_i, b_i)}$ traverses $[y_1, y_2]$ exactly once.  Let $a_i < c_1 < c_2 < b_i$ be such that $\gamma_n(c_1) = y_1$ and $\gamma_n(c_2) = y_2.$   Now we estimate: for $x \in [a_i,b_i),$
\begin{align*}
p^{n-1} \tau (x) - \kappa  &<  \gamma'_n(x) <  p^{n-1} \tau(x) + \kappa \\
(p^{n-1} (\tau(a_i) - \e^2) - \kappa) &<  \gamma'_n(x) <  (p^{n-1} (\tau(a_i) + \e^2) + \kappa).
\end{align*}
Integrating over the intervals $[a_i, b_i)$ and $[c_1, c_2)$ yields
\EQ
(p^{n-1}(\tau(a_i) - \e^2) - \kappa) (b_i-a_i) <  & 1 & < (p^{n-1} (\tau(a_i) + \e^2) + \kappa)(b_i-a_i) \\
(p^{n-1} (\tau(a_i) - \e^2) - \kappa)(c_2-c_1) < & y_2 - y_1 & < (p^{n-1} (\tau(a_i) + \e^2) + \kappa)(c_2-c_1)
\EQE
It follows that
\EQ
\frac{(p^{n-1} (\tau(a_i) - \e^2) - \kappa)(c_2-c_1)}{(p^{n-1} (\tau(a_i) + \e^2) + \kappa)(b_i-a_i)}
< & y_2 - y_1 &
< \frac{(p^{n-1} (\tau(a_i) + \e^2) + \kappa)(c_2-c_1)}{(p^{n-1}(\tau(a_i) - \e^2) - \kappa)(b_i-a_i)}. \\
\frac{(p^{n-1} (\e - \e^2) - \kappa)(c_2-c_1)}{(p^{n-1} (\e + \e^2) + \kappa)(b_i-a_i)}
< & y_2 - y_1 &
< \frac{(p^{n-1} (\e + \e^2) + \kappa)(c_2-c_1)}{(p^{n-1}(\e - \e^2) - \kappa)(b_i-a_i)}. \\
(1- 3\e)\frac{c_2-c_1}{b_i-a_i}
< & y_2 - y_1 &
< (1+3\e) \frac{c_2-c_1}{b_i-a_i}. \\
(1- 3\e) \nu_n(R)
< & m(R) &
< (1+3\e)  \nu_n(R),
\EQE
since $\nu_n(R) = c_2-c_1$ and $m(R) = (y_2-y_1)(b_i-a_i).$
If we let $f$ be any $1$-Lipschitz function on $\TT^2$ with $| f | < 1,$  we see that
\[ \left| \int_{\pi^{-1}S'} f d \nu_n - \int_{\pi^{-1}S'} f dm \right| < 3 \e + 2 \delta \text{ \ \ so, \ }
\| \nu_n|_{\pi^{-1}S'} - m|_{\pi^{-1}S'} \|_\star < 5 \e.  \]
Finally, since $\nu_n$ and $m$ both give $\pi^{-1}(S \sm S')$ measure at most $\e,$ we have $\| \nu_n|_{\pi^{-1}S} - m|_{\pi^{-1}S} \|_\star < 7 \e.$  Since this holds for all $n$ larger than $N,$ the theorem is proven.
\end{proof}

The next (trivial) Corollary gives an extremely non-optimal estimate on the size of $S.$  Its value is in providing an easy method for over-estimating the (failure of) equidistribution of the sequence $T^n_\star \mu$ for any differentiable skewing factor $f$ whatsoever.

\begin{corollary} \label{pskewCor1}
The set $S$ on which the conclusions of Theorem \ref{pskewThm} hold satisfies
\[ \mu(S) \geq 1 - m( \{ x : |f'(x) + p \gamma'(x)| <  \sup |f'| \frac{1}{p-1} \} ) =: 1- \beta. \]
In particular, if $\varphi$ is any continuous test function on $\TT^2$ and $\mu$ projects to Haar measure on the first coordinate then
\[ \limsup_{n \to \infty} \left| \int \varphi dm - \int \varphi \circ T^n d\mu \right| \leq \beta \sup | \varphi |. \]
\end{corollary}
\begin{proof}
If $x \notin S$ then $\tau(x) = 0$ (where $\tau$ is defined as in the proof of Theorem \ref{pskewThm}.)  A simple application of the triangle inequality yields
\[ |f'(x) + p \gamma'(x)|  \leq  \sum_{n=1}^\infty | p^{-n} f'(x+n\alpha) | \leq \sup |f'| \frac{1}{p-1}, \]
which proves the first claim.  The second claim is a trivial consequence of the first.
\end{proof}

\begin{corollary} \label{pskewCor2}
Let $(X,T),\gamma$ be as in Theorem \ref{pskewThm} and suppose $\mu$ projects to Lebesuge measure on the first coordinate.  If $f$ and $\gamma$ are analytic, then either $T^n_\star \mu$ tends to Lebesgue measure or $f(x) = \gamma(Tx) - p\gamma(x)$ plus some constant.
\end{corollary}

\begin{proof}
If $f$ and $\gamma$ are analytic, then so are $f, \gamma',$ and $\tau.$  If the set of zeros of $\tau$ is a null set, then by Theorem \ref{pskewThm} $T^n_\star \mu$ tends toward Lebesgue measure.  Otherwise, analyticity of $\tau$ implies $\tau \equiv 0.$  The relation
\[ \tau(x) =  p \gamma'(x) + f'(x) + p^{-1}( \tau(Tx) - p \gamma'(Tx)) \]
reduces to $f'(x) = \gamma'(Tx) - p\gamma'(x).$
\end{proof}

\begin{example}
Let $\mu$ be Lebesgue measure on the horizontal line $(x,y_0)$ and let $T(x,y) = (x+\alpha, 2y+f(x)).$  If $f$ is strictly monotonic then $|\tau(x)|>0.$  So, by Theorem \ref{pskewThm}, $T^n_\star \mu$ equidistributes.\end{example}

\begin{example} \label{nontrivialMeasureOnPONE}
Let $f$ be arbitrary and let $\gamma'$ agree with
\[ -\sum_{n=0}^\infty p^{-n-1} f'(x + n\alpha) \text{ \ on \ } [0,1/2] \]
and disagree elsewhere.  Then $S=(1/2,1).$  Let $\mu$ be the measure on the graph of $\gamma$ which projects to Lebesgue measure on the first coordinate.  Let $\theta$ be any weak$^\star$ limit of $T^n_\star \mu$ and apply Theorem \ref{pskewThm} to conclude that $\theta$ is equal to Lebesgue measure on a translate $\pi^{-1}(\beta + S)$ and is supported on a $\frac{p}{p-1} \sup | f' |$-Lipschitz function elsewhere.
The same is true of $T^n_\star \theta$ for all $n.$

Because of the shape of the support of $\theta$ on $\pi^{-1}(\beta + S),$ there is some constant $\e$ (depending only on the Lipschitz constant $\frac{p}{p-1} \sup | f' |$) such that one can always find a ball of radius $\e$ disjoint from the support.  It is easy to see that there is some smaller constant $\e'>0$ such that $\| T^n_\star \theta - m \|_\star > \e' $ for all $n.$  So we can choose a weak$^\star$ limit of 
$\frac{1}{N} \sum_{n=0}^{N-1} T^n_{\star \star} \delta_\theta = \frac{1}{N} \sum_{n=0}^{N-1} \delta_{T^n_\star \theta}$
to get an invariant measure $\eta \in P(P_1)$ which must be different from $\delta_m.$

Now take $f$ to be the function $f(x) = x.$  In this case, further analysis, reveals that on the complement of $\pi^{-1}( \beta + S), \theta$ is supported on a line with slope $\frac{p}{p-1}.$ $T^n_\star \theta$ also has this property.  So $T^n_\star \theta = (n \alpha, y_n) +\theta$ for some $y_n \in \TT$ (here, $(n \alpha, y_n)+ \theta$ represents pushforward of $\theta$ under addition of $(n \alpha, y_n)$.)  It follows that the orbit closure $Y := \overline{ \{ T^n_\star \theta : n \} }$ is homeomorphic to some closed subset of $\TT^2.$  In fact, suppose $(\beta, y_0)$ is one endpoint of the line which is the support of $\theta$ on the complement of $\pi^{-1}(\beta + S).$  Then $(Y,T_\star)$ is isomorphic as a topological system to the orbit closure of this point $\overline{ \{ T^n(\beta, y_0) : n \in \ZZ \} }$ with isomorphism given by $(x,y) \mapsto (x-\beta,y-y_0) + \theta.$

The system $(Y, T_\star)$ has, as a factor, the rotation by $\alpha$ on $\TT.$  The factor map is given by $Y \owns (x,y) + \theta \mapsto x.$  So, the measure $\eta \in P(Y) \subseteq P(P_1)$ we constructed above must project to the Lebesgue measure on the rotation factor.  This  provides an example of an invariant measure on $P_1$ which is not a convex combination of delta-masses at fixed points.
\end{example}

Acknowledgements:
We would like to thank Vitaly Bergelson for supervising this project and Manfred Einsiedler, who's simple demonstration that the curve $(x,x^2)$ equidistributes on $\TT^2$ provided the original inspiration.  At least as significantly we thank Hillel Furstenberg who's work heavily influenced the author (in particular his work on unique ergodicity of skew products \cite{furstenberg}.)  Additionally we wish to thank John Greismer, Cory Christopherson, Michael Bj\"orklund and especially Roger Z\"ust, with whom the author had many fruitful discussions during the writing process.  Most importantly, Sasha Leibman's patient reading revealed many non-trivial mistakes and contributed greatly to the quality and validity of this paper.

\
\bibliographystyle{amsplain}

\end{document}